\documentclass[11pt,reqno]{amsart}
 \usepackage{amsgen, amstext,amsbsy,amsopn, amsthm, amsfonts,amssymb,amscd,amsmath,euscript,enumerate,url,verbatim,calc,xypic,tikz}
\usepackage{hyperref}
\usepackage{MnSymbol}
\usepackage{pgfkeys}
\usepackage{mathtools}
\usepackage{eqnarray,amsmath}

\def\multiset#1#2{\ensuremath{\left(\kern-.3em\left(\genfrac{}{}{0pt}{}{#1}{#2}\right)\kern-.3em\right)}}

\usetikzlibrary{arrows}
\newcommand{\midarrow}{\tikz \draw[-triangle 90] (0,0) -- +(.1,0);}
 \usepackage{latexsym}
 \usepackage{graphics}
 \usepackage{color}
\usepackage{lastpage}
\usepackage{fancyhdr}
\usepackage{multirow}
\allowdisplaybreaks
\usepackage{graphicx}
\graphicspath{ {F:/IMAGES/} }

  \newcommand{\Ass}{\operatorname{Ass}}

 \newcommand{\supp}{\operatorname{supp}}
 
\newcommand{\rad}{\operatorname{rad}}

   \newcommand{\lcm}{\operatorname{lcm}}

\newcommand{\sdefect}{\operatorname{sdefect}}

\newcommand{\proset}{\,\mathrel{\lower 4pt\hbox{$\scriptscriptstyle/$}
\mkern -14mu\subseteq }\,} 

 \newtheorem{theorem}{Theorem}[section]
 \newtheorem{corollary}[theorem]{Corollary}
 \newtheorem{lemma}[theorem]{Lemma}
 \newtheorem{proposition}[theorem]{Proposition}

\usepackage{amsmath}
\newtheorem{notation}[theorem]{Notation}

 \theoremstyle{definition}
 
 \newtheorem{remark}[theorem]{Remark}
 \newtheorem{definition}[theorem]{Definition}

\title[Comparing  symbolic powers of edge ideals  of  weighted oriented graphs] {Comparing  symbolic powers of edge ideals  of weighted oriented graphs}
\author[M. Mandal and D.K. Pradhan ]{Mousumi Mandal$^*$ and Dipak Kumar Pradhan }

 \thanks{$^*$ Supported by SERB(DST) grant No.: $\mbox{MTR}/2020/000429$, India}
\thanks{AMS Classification 2010: 05C22, 05C25, 05C38, 05E40}
\address{Department of Mathematics, Indian Institute of Technology Kharagpur, 721302, India} \email{mousumi@maths.iitkgp.ac.in}
\address{Department of Mathematics, Indian Institute of Technology Kharagpur, 721302, India}\email{dipakkumar@iitkgp.ac.in}    

\begin{document}
\maketitle

\begin{abstract}    
Let $D$ be a weighted oriented graph  and $I(D)$ be its edge ideal. If $D$  contains an induced odd cycle    of length $2n+1$, under certain condition,  we show that   $ {I(D)}^{(n+1)} \neq {I(D)}^{n+1}$. We give necessary and sufficient condition for the equality of ordinary and symbolic powers of edge ideal of  weighted oriented graph having each  edge  in some induced  odd cycle of it. We characterize the  weighted naturally oriented unicyclic graphs with  unique odd cycles and weighted naturally oriented even cycles for the equality of ordinary and symbolic powers of their edge ideals. Let $ D^{\prime} $ be the  weighted oriented graph obtained from  $D$   after replacing the weights of vertices with non-trivial weights which are sinks, by  trivial weights. We show that the symbolic powers of $I(D)$ and $I(D^{\prime})$ behave in a similar way.
Finally,  if $D$ is any weighted oriented star graph, we prove that $ {I(D)}^{(s)} =   {I(D)}^s $ for all $s \geq 2.$

\end{abstract}

\section{Introduction}
 Let $k$ be a field and $R=k[x_1,\ldots ,x_n]$ be a polynomial ring in $n$ variables. Let $I$ be a homogeneous ideal of $R$. Then for $s\geq 1$, the $s$-th symbolic power of $I$ is defined as $I^{(s)}=\displaystyle{\bigcap_{P\in \Ass I}(I^sR_P\cap R)}$. Geometrically, the symbolic powers are important since they capture, all the polynomials that vanish with a given multiplicity. We refer  \cite{huneke} to the reader to analyse the background results of symbolic powers of ideals.  By definition it is clear that $I^s\subseteq I^{(s)}$ for all $s\geq 1$ but the reverse containment may fail. It is always an interesting problem to find the necessary and sufficient condition for holding the reverse containment. There is no such criteria to be known when the equality $I^s = I^{(s)}$ holds for any arbitrary ideal. But for certain classes of ideals such as prime or radical ideals, there are equivalent conditions given by  Hochster in  \cite{hochster} and by  Li and Swanson in \cite{li}.
 In this paper we compare the ordinary and symbolic powers of edge ideals of weighted oriented graphs.  
\vspace*{0.1cm}\\
Let $D = (V (D), E(D), w)$ be a weighted oriented graph with the vertex set $V(D)=\{x_1,\ldots,x_n\},$  the edge set $E(D)$ consists of ordered pairs of the form $(x_i,x_j)$ which represents a directed edge from the vertex $x_i$ to the vertex  $x_j$   and the weight function $ w : V (D) \longrightarrow \mathbb N$.    
The weight of a vertex $x_i\in V(D)$ is $w(x_i)$  denoted by $w_i$ or $ w_{x_i}.$ If a vertex $x_i$ of $D$ is a source (i.e., has only arrows leaving $x_i$), we set $w_i = 1$. The edge ideal of $D$ is denoted by $I(D)$ and is defined as 
$I(D)=(x_ix_j^{w_j}|(x_i,x_j)\in E(D)).$ Let $G=(V(G),E(G))$ be the underlying graph of $ D $ whose
vertex set is $V(G) = V(D) $ and edge set is $E(G)= \{ \{x_i, x_j \}~|~(x_i, x_j) \in E(D)  \}   .$ The edge ideal of $ G $ is 
$I(G) = ( x_ix_j ~|~  \{x_i, x_j\} \in E(G) ).$  
\vspace*{0.1cm}\\ 
In general, even for monomial ideals comparison of ordinary and symbolic powers  is a difficult problem. For simple graphs, by \cite[Theorem 5.9]{simis}, we know that all the ordinary and symbolic powers
of edge ideal  coincide if and only if the graph is bipartite. But there is no such result for weighted oriented graphs. As the edge ideal of weighted oriented graph depends upon both the orientation of edges and weights on vertices, it is actually difficult to get the necessary and sufficient condition, even for a particular class of weighted oriented graphs. Recently in \cite{kanoy}, the authors characterize the 
weighted oriented complete graphs and weighted oriented complete bipartite graphs for the equality of all the ordinary and symbolic powers of their edge ideals. If a simple graph $G$ contains an induced odd cycle of length $2n+1$, we know that    $ {I(G)}^{(n+1)} \neq {I(G)}^{n+1} $.  We prove that the same is true for   weighted oriented graphs, if $D$  contains an induced odd cycle $   D^{\prime}     $    of length $2n+1$ with the condition 	`` $ V (D^{\prime}) \setminus  
N_D^+(V^{+}(D))$ contains one vertex which is not source in $D^{\prime}$, otherwise, it contains a  vertex which is source  in  $D^{\prime}$ with trivial weight in $D$",    then  $ {I(D)}^{(n+1)} \neq {I(D)}^{n+1}$ (see Proposition \ref{oddcycle}).  As one of its application, we characterize  the  weighted   oriented graph having each  edge  in some induced  odd cycle of it, for the equality of ordinary and symbolic powers of its edge ideal in Theorem \ref{oddcycle2}.    Also we characterize the  weighted naturally oriented unicyclic graphs with  unique odd cycles and weighted naturally oriented even cycles for the equality of ordinary and symbolic powers of their edge ideals (see Theorem \ref{unicyclic}, Corollary \ref{evencycle1} and Proposition \ref{evencycle2}, respectively).
\vspace*{0.1cm}\\
Let $ D^{\prime} $ be the  weighted oriented graph obtained from  $D$   after replacing the weights of vertices with non-trivial weights which are sinks, by  trivial weights. If we assume all vertices of $D$ with non-trivial weights are sinks, then $ D^{\prime} = G $ and in \cite{mandal1}, we proved that the symbolic powers of $I(D)$ and $I(G)$ behave in a similar way. In this paper, even if we do not assume all vertices of $D$  with non-trivial weights are sinks, we show that the symbolic powers of $I(D)$ and $I(D^{\prime})$  behave in a similar way (see Theorem \ref{sym.theorem.1}). We prove that the symbolic defects of edge ideals of $D$ and $D^{\prime}$ are same in Proposition \ref{sdefect} and so we have ${{I(D^{\prime})}}^{(s)} = {I(D^{\prime})}^s $  if and only if $ {I(D)}^{(s)} = {I(D)}^s $ for each $s \geq 1.$ As an instant application of this result, we characterize
the weighted oriented even cycles of length $ 4 $ which are not naturally oriented for the equality of all the ordinary and symbolic powers of their edge ideals (see Proposition \ref{evencycle4}).
Finally  in Theorem \ref{stargraph},  we prove the equality of  ordinary and
symbolic powers of edge ideal of any weighted oriented star graph.

\section{Preliminaries}

In this section, we recall some definitions and results for the weighted oriented graphs.
\begin{definition}
	A vertex cover $ C $ of $ D $ is a subset of $  V(D)  $ such that if $ (x, y) \in E(D) ,$ then
	$ x \in C $ or  $ y \in C . $ A vertex cover $ C $ of $ D $ is minimal if each proper subset of $ C $ is not a
	vertex cover of $ D. $ We set $(C)$ to be the ideal generated by the vertices of $C.$
\end{definition}
\begin{definition}
	Let $ x $ be a vertex of a weighted oriented graph $ D, $ then the sets $ N_D^+ (x) =	\{y ~|~ (x, y) \in E(D)\}  $ and  $ N_D^- (x) = \{y ~|~ (y, x) \in E(D)\}  $ are called the out-neighbourhood and the in-neighbourhood of $ x ,$ respectively. Moreover, the neighbourhood of $ x $ is the set $ N_D(x) = N_D^+ (x)\cup N_D^- (x) .$ For a subset of the vertices
	$ W \subseteq  V (D),  $  we define $ N_D^+(W) $, $ N_D^-(W) $ and $ N_D(W) $ similarly. For $T \subset V(D) ,$ we
	define the  induced  subgraph $\mathcal{D} = (V( \mathcal{D}), E(\mathcal{D}), w)$ of $D$ on $T$ to be the
	weighted oriented graph such that $V (\mathcal{D}) = T$ and for any
	$ u, v \in V (\mathcal{D}),  $  $ (u,v) \in E(\mathcal{\mathcal{D}})$  if and only if
	$(u,v) \in E(D)$.
	Here $ \mathcal{D} = (V (\mathcal{D}), E(\mathcal{D}), w) $
	is a weighted oriented graph with the same orientation  as in $D$ and for any $ u  \in V (\mathcal{D}), $
	if $ u $ is not a source in $ \mathcal{D}, $ then its weight equals to the weight of $ u $ in $D,$ otherwise, its
	weight in $ \mathcal{D} $ is $ 1. $ For a subset $W \subset V(D) $ of the vertices in $ D, $ define $ D \setminus  W $ to
	be the  induced subgraph of $ D $ with the vertices in $ W $ (and  their    incident edges) deleted.  Define $\deg_D(x) = |N_D(x)|$ for  $ x \in V(D) $. 
	 A vertex $ x \in V(D) $ is called a source vertex if $N_D (x)= N_D^+ (x) .$ A vertex $ x \in V(D) $ is called a sink vertex if $N_D (x)= N_D^- (x) .$
We set $V^+(D)$ as the set of vertices of $D$ with non-trivial weights.

\end{definition}
\begin{definition}\cite[Definition 4]{pitones}
Let $ C $ be a vertex cover of a weighted oriented graph $ D.$ We define \vspace*{0.2cm}\\
\hspace*{3cm}$ L_1^D(C) = \{x \in C ~|~ N_D^+ (x) \cap C^c \neq        \phi \}, $ \vspace*{0.2cm}\\
\hspace*{2.85cm} $L_2^D(C) = \{x \in C ~|~x\notin L_1^D(C) ~\mbox{and}~  N_D^-(x) \cap C^c \neq   \phi \}$ and  \vspace*{0.2cm}\\ 
\hspace*{2.8cm}  $ L_3^D(C) = C \setminus (L_1^D(C) \cup L_2^D(C))$ \vspace*{0.2cm}\\
 where $ C^c $ is the complement of $  C ,$ i.e., $ C^c = V(D) \setminus C. $
\end{definition}
\begin{lemma}\cite[Proposition 6]{pitones}\label{s.v.0}
Let $ C $ be  a  vertex  cover  of $ D. $  Then $ L_3^D(C) =\phi $ if  and  only  if $ C $ is  a minimal vertex cover of $ D .$  	

\end{lemma}
 
 \begin{lemma}\cite[Proposition 5]{pitones}\label{L3}  
 If   $ C $ is a vertex cover of $ D, $ then $ L_3^D(C) =\{x\in  C~|~N_D(x) \subset C\}.	 $  
 \end{lemma}

\begin{definition}\cite[Definition 7]{pitones}
A vertex cover $ C $ of $ D $ is strong if for each $ x \in L_3^D(C)  $ there is $ (y, x) \in 
E(D) $ such that $ y \in L_2^D(C) \cup L_3^D(C)$ with $  y \in V^+(D)$   (i.e., $ w(y) \neq  1 $).   
\end{definition}

\begin{remark}\cite[Remark 8, Proposition 5]{pitones}\label{s.v.1}
	A vertex cover	$ C $ of $D$  is strong if and only if for each $ x \in L_3^D(C),$   we have $N_D^{-}(x) \cap V^+(D) \cap [C\setminus {L_1^D{(C)}} ] \neq \phi.$
\end{remark}

\begin{lemma}\cite[Corollary 9.]{pitones}\label{minimal to strong}
	  If $ C $ is a minimal vertex cover of $ D, $ then $ C $ is strong.
\end{lemma}

\begin{definition}
A strong vertex cover $C$ of $D$ is said to be a maximal strong vertex cover of $D$ if it is not contained in any  other strong vertex cover of $D.$
\end{definition}


\begin{definition}\cite[Definition 32]{pitones}
	A weighted oriented graph $ D $ has the minimal-strong property if each
	strong vertex cover is a minimal vertex cover.

\end{definition}
\begin{lemma}\cite[Lemma 47]{pitones}\label{s.v.4}
	If the vertices of $ V^+(D) $ are sinks, then $ D $ has the minimal-strong property. 	
\end{lemma}

\begin{definition}\cite[ Definition 19]{pitones}\label{definition19}   
Let $ C $ be a vertex cover of $ D. $  The irreducible ideal associated to $ C $ is the ideal 
$ I_C =:(L_1^D(C)\cup \{x_j^{w(x_j)} |x_j \in  L_2^D(C) \cup  L_3^D(C)    \}).$	
\end{definition}

\begin{lemma}\cite[Lemma 20]{pitones}\label{edge}
	Let $ D $ be a  weighted oriented graph. Then $I(D) \subseteq I_C$, for each vertex cover $ C $ of $ D $.  	
\end{lemma}

The next lemma describes the irreducible decomposition of the edge ideal of a  weighted oriented graph $ D $.

\begin{lemma}\cite[Theorem 25, Remark 26]{pitones}\label{s.v.2}
Let $ D $ be  a weighted oriented graph and $C_1,\ldots, C_s$ are the strong vertex covers of $ D ,$ then  the irredundant irreducible decomposition of $ I(D) $ is
    $$I(D) =  I_{C_1} \cap\cdots\cap I_{C_s} $$ 
where each $ I_{C_i} = ( L_1^D(C_i) \cup \{x_j^{w(x_j)}~|~x_j \in L_2^D(C_i) \cup L_3^D(C_i)\} ) ,$   $ \rad(I_{C_i})=P_i = (C_i)$.
\end{lemma}
\begin{corollary}\cite[Remark 26]{pitones}\label{s.v.3}
	Let $ D $ be a weighted oriented graph. Then $ P $ is an associated 
	prime of $ I(D) $ if and only if $ P = (C) $ for some strong vertex cover $ C $ of $ D. $
\end{corollary}

Let $ I \subset R$ and $ I = Q_1\cap \cdots \cap Q_m $ be the primary decomposition of ideal $I$. For $ P \in \Ass(R/I), $
we denote $ Q_{\subseteq P} $ to be the intersection of all $ Q_i $ with 
$ \sqrt{Q_i} \subseteq P.  $ If $C$ is a  strong vertex cover of a weighted oriented graph $ D $, then $(C)  \in \Ass(R/I(D))$. 
We denote   $ I_{\subseteq {C}} $ as  $ I_{\subseteq {(C)}} $. In the following lemma, we write the \cite[Theorem 3.7]{cooper} for edge ideals of weighted oriented graphs.
\begin{lemma}\cite[Theorem 3.7]{cooper}\label{cooper} Let $I$ be the edge ideal of a weighted oriented graph $ D $ and $C_1,\ldots,C_r$ are the maximal strong vertex covers of $D.$ Then  $$I^{(s)}=(I_{\subseteq {C_1}})^s \cap\cdots\cap (I_{\subseteq {C_r}})^s.$$ 
\end{lemma}

\begin{lemma}\cite[Lemma 3.1]{mandal1}\label{mandal}
	Let $ D $ be a weighted oriented graph. If $ V(D) $ is a strong vertex cover of $ D ,$ then $ I(D)^{(s)}=I(D)^s $ for all $ s \geq 2.$  	
\end{lemma}  

\begin{lemma}\cite{kanoy}\label{kanoy}
	Let $ D $ be a weighted oriented graph. Then $ V (D) $ is a strong vertex cover if and
	only if $ N^+_D (V^+(D)) = V (D). $
\end{lemma}

 \begin{lemma}\cite[Corollary 3.8]{mandal1}\label{DG}
	Let $ D $ be a weighted oriented bipartite graph  where the vertices of   $V^{+}(D)$ are sinks. Then 
	$ {{I(D)}^{(s)}} = {{I(D)}^{s}} $ for all $ s \geq 2. $  	
\end{lemma} 

\begin{definition}  
	Let $ I \subset R $ be a monomial ideal. Let $ \mathcal{G}(I) $ be the set of minimal
	generators of the ideal $ I $. Let $ J $ be the ideal we need to add to $ I^s $ to achieve $ I^{(s)}, $ i.e.,
	$ I^{(s)} = I^s + J $. We set $ \sdefect(I,s) $ is the number of elements of $ \mathcal{G}(J) $.	
\end{definition}

The following  lemma based on the extension of  ideals.

\begin{lemma}\cite{atiyah}\label{atiyah}
Let $ f: A  \longrightarrow  B $ be a ring homomorphism. Let $I$ be an ideal of $A.$ The extension $I^e$ of  $I$  is the ideal $Bf(I)$  generated by $f(I)$ in $B$. If $I_1$ and $I_2$ are ideals of $ A $,   then	

\begin{enumerate}
	\item[(a)] $(I_1 \cap I_2)^{e} \subseteq (I_1)^{e} \cap (I_2)^{e}$,
	\item[(b)] $(I_1 I_2)^{e} = (I_1)^{e}  (I_2)^{e}$.  
\end{enumerate}

\end{lemma}

\begin{notation}
Let  $g \in  k[x_1,\ldots,x_n]$ be a monomial. We define support of $g$ $=  \{x_i:x_i \divides g\}  $ and we denote it by $\supp(g).$

\end{notation}

%
%

\section{Comparing ordinary and symbolic powers of weighted oriented graphs}
In this section, we compare the ordinary and symbolic powers of edge ideals of weighted oriented graphs containing  induced  odd cycles  and weighted naturally oriented even cycles.


In \cite[Proposition 4.10]{huneke}, if a simple graph contains an induced odd cycle $ C_{2n+1} = (x_1,\ldots,x_{2n+1}) $, the authors have shown that the $(n+1)-$th ordinary and symbolic power of its edge ideal are different. In this paper we  extend this  result for
weighted oriented graphs under certain condition.

\begin{proposition}\label{oddcycle}
	Let $ D $ be a  weighted oriented graph.
	Let $ D^{\prime} $ be an induced  odd cycle with underlying graph  $ C_{2n+1} = (x_1,\ldots,x_{2n+1}) $ where $ V (C_{2n+1}) \nsubseteq
	N_D^+(V^{+}(D))$ and it satisfies the condition
	`` $ V (C_{2n+1}) \setminus  
	N_D^+(V^{+}(D))$ contains one vertex which is not source in $D^{\prime}$, otherwise, it contains a  vertex which is source  in  $D^{\prime}$ with trivial weight in $D$". 
	Then $ I(D)^{(n+1)} \neq I(D)^{n+1}. $	    
\end{proposition}

\begin{proof}
	
	Let $w_i = w(x_i)$ for  $x_i \in V(C_{2n+1})$.  Let $f={x_1}^{a_1}\cdots{x_{2n+1}}^{a_{2n+1}}$ where each  $a_i=w_i$ if $ N^-_{D^{\prime}}(x_i)  \neq  \phi$ (i.e., $x_i$ is not source in $D^{\prime}$)  and $a_i=1$ if $ N^-_{D^{\prime}}(x_i)  =  \phi$ (i.e., $x_i$ is source in $D^{\prime}$). We claim that $ f \in   I(D)^{(n+1)} \setminus I(D)^{n+1}. $ We set $ m_{\{x_i,x_j\}} $ as the the minimal generator of $ I(D^{\prime})$ corresponding to the edge
	$ \{x_i, x_j\} \in E(C_{2n+1}) $. Note that $ m_{\{x_i,x_j\}} $ can be $ x_ix_j^{w_j} $ or $ x_jx_i^{w_i} $ and $x_i^{a_i} x_j^{a_j}$ is multiple of  $ m_{\{x_i,x_j\}} $. Let $ u \in  V (C_{2n+1}) \setminus N_D^+(V^+(D))$ be that vertex which is not source in $D^{\prime}$ and  if $u$ is source  in  $D^{\prime}$, then its weight is $1$ in $D$.
	Without loss of generality
	we can assume that $ u = x_1$. Here $N_{D^{\prime}}(x_1)= \{x_2, x_{2n+1}\}.$

	
	%

	Let $ C $ be a maximal strong vertex cover of $ D $.
	Suppose  $x_1 \notin C$. Then for any   strong vertex cover $C^{\prime} \subseteq C,$ $x_1 \notin C^{\prime}$ and so $x_2$ and $x_{2n+1} \in C^{\prime}$. If $(x_2,x_1) \in E(D^{\prime}),$ then for each strong vertex cover $C^{\prime} \subseteq C,$ $ x_1 \in  N^+_{D}(x_2) \cap {C^{\prime}}^c$ and hence $x_2 \in L_1^D(C^{\prime})$. This implies $x_2 \in I_{\subseteq C}.$ If $(x_1,x_2) \in E(D^{\prime}),$ then $x_2^{a_2} =   x_2^{w_2} \in I_{\subseteq C}.$ In both cases $x_2^{a_2} \in I_{\subseteq C}.$ By the same argument we can show $x_{2n+1}^{a_{2n+1}} \in I_{\subseteq C}.$
	By Lemma \ref{edge}, $m_{\{x_3,x_4\}},\ldots,m_{\{x_{2n-1},x_{2n}\}} \in I_{\subseteq C}$. So $x_{2}^{a_2}\cdot x_{2n+1}^{a_{2n+1}}\cdot m_{\{x_3,x_4\}}\cdots  m_{\{x_{2n-1},x_{2n}\}} \in ({I_{\subseteq C}})^{n+1}.$ Hence $f={x_1}^{a_1}\cdots{x_{2n+1}}^{a_{2n+1}}\in ({I_{\subseteq C}})^{n+1}$. Suppose $x_1 \in C$. By Remark \ref{s.v.1}, we have $x_1 \notin L^{D}_3(C)$ and by Lemma \ref{L3}, at least one element of $N_D(x_1)$ does not belong to $C$.  
	Then for any   strong vertex cover $C^{\prime} \subseteq C,$ at least one element of $N_D(x_1)$ does not belong to $C^{\prime}$ and so  $x_1 \in C^{\prime}$. If $x_1$ is not source in $D^{\prime}$,  $x_1^{a_1}  =  x_1^{w_1} \in I_{\subseteq C}.$ If $x_1$ is source in $D^{\prime}$, by our assumption  $w_1 = 1 $ in $D$ and hence $x_1 \in I_{\subseteq C}.$ In both cases $x_1^{a_1} \in I_{\subseteq C}.$ By Lemma \ref{edge}, $m_{\{x_2,x_3\}},\ldots,m_{\{x_{2n},x_{2n+1}\}} \in I_{\subseteq C}$. So $x_1^{a_1}\cdot m_{\{x_2,x_3\}}\cdots   m_{\{x_{2n},x_{2n+1}\}} \in ({I_{\subseteq C}})^{n+1}.$ Hence $f={x_1}^{a_1}\cdots{x_{2n+1}}^{a_{2n+1}}\in ({I_{\subseteq C}})^{n+1}$. Similarly for any  maximal strong vertex cover $ C $  of $ D $, we can show $f={x_1}^{a_1}\cdots{x_{2n+1}}^{a_{2n+1}}\in ({I_{\subseteq C}})^{n+1}$. Hence $f \in I(D)^{(n+1)}$. It remains to show that $f \notin I(D)^{n+1}.$ Since $ \supp(f) = V(D^{\prime})$  and $D^{\prime}$ is an induced subgraph of $D$, it is enough to show $f \notin I(D^{\prime})^{n+1}.$ Here $|\supp(f)| = 2n+1.$   Thus if we want to express $f$ as a multiple of product of some $n+1$ minimal generators of $I(D^{\prime}),$ then one $ x_i^{a_i} $ of $f$  must involve in two minimal generators of $I(D^{\prime})$. But by definition of $a_i,$ any $ x_i^{a_i} $ of $f$  can not involve in two minimal generators of $I(D^{\prime})$. 	Therefore $f={x_1}^{a_1}\cdots{x_{2n+1}}^{a_{2n+1}} \notin I(D^{\prime})^{n+1}.$ Thus $f \in I(D)^{(n+1)} \setminus I(D)^{n+1}$. Hence the proof follows.
\end{proof}	

\begin{remark}
	The above proposition may not be true if we remove the given condition.

\begin{figure}[!ht]
		\begin{tikzpicture}[scale=1]
		\begin{scope}[ thick, every node/.style={sloped,allow upside down}] 
		\definecolor{ultramarine}{rgb}{0.07, 0.04, 0.56} 
		\definecolor{zaffre}{rgb}{0.0, 0.08, 0.66}
		
		\draw [fill, black](7.5,-0.75) --node {\midarrow}(7.5,0.75);  
		\draw [fill, black](9,0) --node {\midarrow}(7.5,-0.75);
		\draw [fill, black](9,0) --node {\midarrow}(7.5,0.75);
		\draw[fill, black]  (10.5,0) --node {\midarrow}(9,0);

		\draw[fill] [fill] (7.5,-0.75) circle [radius=0.04];
		\draw[fill] [fill] (9,0) circle [radius=0.04];
		\draw[fill] [fill] (10.5,0) circle [radius=0.04];
		\draw[fill] [fill] (7.5,0.75) circle [radius=0.04];

		\node at (7.5,-1.3) {$x_2$};
		\node at (7.5,1.3) {$x_3$};
		\node at (9.1,0.5) {$x_{1}$};  
		\node at (10.5,0.5) {$x_{4}$};
		
		\node at (7,-0.9) {$1$};
		\node at (7,0.9) {$3$};
		\node at (9.1,-0.5) {$3$};  
		\node at (10.5,-0.5) {$1$};

		\draw [fill, black](13,-0.75) --node {\midarrow}(13,0.75);  
		\draw [fill, black](14.5,0) --node {\midarrow}(13,-0.75);
		\draw [fill, black](14.5,0) --node {\midarrow}(13,0.75);

		\draw[fill] [fill] (13,-0.75) circle [radius=0.04];
		\draw[fill] [fill] (14.5,0) circle [radius=0.04];
		\draw[fill] [fill] (13,0.75) circle [radius=0.04];

		\node at (13,-1.3) {$x_2$};
		\node at (13,1.3) {$x_3$};
		\node at (14.6,0.5) {$x_{1}$};

		\node at (12.5,-0.9) {$1$};
		\node at (12.5,0.9) {$3$};
		\node at (14.6,-0.5) {$1$};  	
		
		\node at (8.6,-1.6) {$D$};
		\node at (13.7,-1.6) {$D^{\prime}$};	

		\end{scope}
		\end{tikzpicture}
		\caption{A weighted oriented graph $D$ containing an induced odd  cycle  $D^{\prime}$ of length $3$.}\label{fig.3}
	\end{figure}
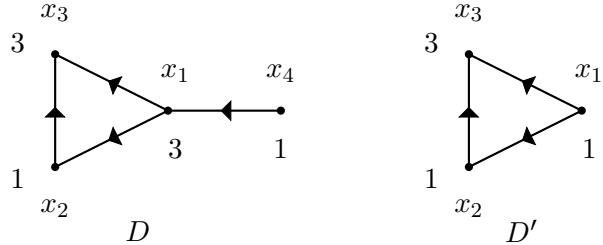  
 
 \hspace*{0.5cm}For example consider the  weighted oriented graph $D$ as in Figure \ref{fig.3}. Then $I(D) = (x_1x_2,x_2x_3^3,x_1x_3,x_4x_1^3)$. Here $D$ contains an induced  cycle  $D^{\prime}$ of length $3$ where $ x_2\in  N_D^+(V^+(D)) $,  $x_3  \in  N_D^+(V^+(D))$ and   $ x_1 \in  V (D^{\prime}) \setminus N_D^+(V^+(D))$. Note that $x_1$ is  source in $D^{\prime}$ but $w(x_1) \neq 1$ in $D$. 
Using Macaulay $ 2 $, we see that $I(D)^{(2)} = I(D)^2$. 	
	
\end{remark}


Now we see some applications of Proposition \ref{oddcycle} to  weighted oriented graphs containing  induced  odd cycles.

\begin{theorem}\label{oddcycle2}
	Let $ D $ be a weighted oriented graph such that each edge of $ D $ lies in some induced   odd cycle of it.
	Then $ V(D) $ is a strong vertex cover of $ D $ if and only
	if $ {I(D)}^{(s)} = {I(D)}^s $ for all $ s \geq 2. $

\end{theorem}

\begin{proof}
	If  $ V(D) $ is a strong vertex cover of $ D $, then by Lemma \ref{mandal}, $ {I(D)}^{(s)} = {I(D)}^s $ for all $ s \geq 2. $

	Assume that $ {I(D)}^{(s)} = {I(D)}^s $ for all $ s \geq 2. $ Suppose $ V (D) $ is not a strong vertex cover. By Lemma \ref{kanoy}, we have $ N^+_D (V^+(D)) \neq V (D). $ Let $u \in   V(D)  \setminus N^+_D (V^+(D)).$
	
	\textbf{Case (1)} Suppose $u$ is not source in $D$. 
	
	Since $u$ is not source in $D$,  there exists some induced  odd cycle $D^{\prime}$ of $D$ such that $u  $ is not  source in $D^{\prime}$. Here $u \in   V(D^{\prime})  \setminus N^+_D (V^+(D))$. If $|V(D^{\prime})| = 2m+1$ for some $m$, then by Proposition \ref{oddcycle}, we get $ I(D)^{(m+1)} \neq I(D)^{m+1},$ which is a contradiction.

	\textbf{Case (2)} Suppose $u$ is source in $D$. 
	
	Then consider any
	induced  odd cycle $D^{\prime\prime}$ of $D$ containing the vertex $u$. Here $u$ is source in $D^{\prime\prime}$ and $w(u)=1$ in $D$. Here $u \in   V(D^{\prime\prime})  \setminus N^+_D (V^+(D))$. If $|V(D^{\prime\prime})| = 2k+1$ for some $k$, then by Proposition \ref{oddcycle}, we get $ I(D)^{(k+1)} \neq I(D)^{k+1},$ which is a contradiction.  	
\end{proof}

The following result is an immediate consequence of the above result.
\begin{corollary}\label{oddcycle3}
	Let $ D $ be a weighted oriented odd cycle.
	Then $ V(D) $ is a strong vertex cover of $ D $ if and only
	if $ {I(D)}^{(s)} = {I(D)}^s $ for all $ s \geq 2. $

\end{corollary}


\begin{corollary}\label{clique.oddcycle.completegraph}
	Let $ D $ be a weighted oriented graph with underlying graph G is a clique
	sum of finite number of odd cycles and complete graphs. Then $ V(D) $ is a strong vertex cover of $ D $ if and only
	if $ {I(D)}^{(s)} = {I(D)}^s $ for all $ s \geq 2. $

\end{corollary}

\begin{proof}
In a complete graph, each edge of complete graph lies in some induced odd cycle of length $3$.   Thus each edge of $ D $ lies in some induced   odd cycle of it.	Hence the proof follows from Theorem \ref{oddcycle2}.  	
\end{proof}

\begin{corollary}\label{m-partite graph}
	Let $ D $ be a weighted oriented graph with underlying graph G is a complete $m-$partite graph for some $m \geq 3$. Then $ V(D) $ is a strong vertex cover of $ D $ if and only
	if $ {I(D)}^{(s)} = {I(D)}^s $ for all $ s \geq 2. $

\end{corollary}

\begin{proof}
 Note that each edge of $G$ lies in some induced odd cycle  of length $3$. Therefore each edge of $ D $ lies in some induced   odd cycle of it.	Hence the proof follows from Theorem \ref{oddcycle2}.  	
\end{proof}

In the next result we see that presence of  certain induced weighted naturally oriented path guarantees the failure in the equality of $3$rd ordinary and symbolic power of edge ideal of  weighted oriented graph.

\begin{definition}
	A  path is naturally oriented if  all edges of path are oriented in one direction.	
\end{definition}       

\begin{lemma}\label{atmost}  
	Let $ D $ be a  weighted oriented graph such that at most one edge oriented into each vertex.
	Let	$D^{\prime}$ be an induced weighted naturally oriented path  of length $3$ of $D$ with  $V(D^{\prime}) = \{x_{i-1},x_i,x_{i+1},x_{i+2}\}$, $E(D^{\prime}) = \{ (x_{j},x_{j+1}) ~|~ i-1 \leq j \leq  i+1 \}$, $w(x_i) \geq 2$ and $w(x_{i+1}) = 1$.
	Then $ {I(D)}^{(3)} \neq {I(D)}^3 $.           	
\end{lemma}

\begin{proof}
Let $w_j = w(x_j)$ for  $x_j \in V(D^{\prime})$. 	We claim $g= x_{i-1}x_i^{w_i}x_{i+1}^2x_{i+2}^{w_{i+2}}  \in   {I(D)}^{(3)}$. Let $ C $ be a maximal strong vertex cover of $ D $.
	Suppose  $x_{i+2} \notin C$. Then for any   strong vertex cover $C^{\prime} \subseteq C,$ $x_{i+2} \notin C^{\prime}$ and so $x_{i+1}$  $ \in C^{\prime}$. Since $(x_{i+1},x_{i+2}) \in E(D),$ for each strong vertex cover $C^{\prime} \subseteq C,$ $ x_{i+2} \in N^+_{D}(x_{i+1}) \cap  {C^{\prime}}^c$ and hence $x_{i+1} \in L_1^D(C^{\prime})$. This implies $x_{i+1} \in I_{\subseteq C}.$ By Lemma \ref{edge}, $x_{i-1}x_i^{w_i} \in I_{\subseteq C}$. So  $ x_{i-1}x_i^{w_i}.(x_{i+1})^2 \in ({I_{\subseteq C}})^{3}$. Hence $g \in ({I_{\subseteq C}})^{3}$. Suppose $x_{i+2} \in C$. By definition of $D,$  $|N^-_{D}(x_{i+2})|=1$.  Since  $N^-_{D}(x_{i+2}) = \{x_{i+1}\} \nsubseteq V^+(D),$   we have $x_{i+2} \notin L^{D}_3(C)$ and by Lemma \ref{L3}, at least one element of $N_D(x_{i+2})$ does not belong to $C$.  
	Then for any   strong vertex cover $C^{\prime} \subseteq C,$ at least one element of $N_D(x_{i+2})$ does not belong to $C^{\prime}$ and so  $x_{i+2} \in C^{\prime}$. This implies $x_{i+2}^{w_{i+2}}  \in I_{\subseteq C}.$  By Lemma \ref{edge}, $ x_ix_{i+1}    \in I_{\subseteq C}$. So $(x_ix_{i+1})^2.x_{i+2}^{w_{i+2}}  \in ({I_{\subseteq C}})^{3}.$ Hence $g \in ({I_{\subseteq C}})^{3}$.
	Similarly for any  maximal strong vertex cover $ C $  of $ D $, we can prove that $g \in ({I_{\subseteq C}})^{3}$. Therefore $g \in I(D)^{(3)}$. Notice that $g \notin I(D^{\prime})^{n+1}.$ Since $ \supp(g) = V(D^{\prime})$  and $D^{\prime}$ is an induced subgraph of $D$, $g \notin I(D)^{3}.$  
	Hence the result follows.	
\end{proof}

Next we see some applications of Lemma \ref{atmost} to some  weighted oriented graphs. 

\begin{definition}
A  cycle is naturally oriented if  all edges of cycle are oriented in clockwise direction. In a naturally oriented unicyclic graph, the cycle is naturally oriented and each edge of the tree connected with the cycle is oriented away from the cycle.	
\end{definition}

\begin{remark}\label{svc1}    
	Let $D$ be a weighted naturally oriented unicyclic graph. By \cite[Proposition 15]{pitones}, $V(D)$ is a strong vertex cover of $D$ if and only if $D$ is naturally oriented and $w(x) \geq 2$ when $\deg_D(x)\geq 2 $ for all $x \in V(D)$.
\end{remark}

%

In the next result, we characterize the  weighted naturally oriented unicyclic graphs with a unique odd cycles for the equality of ordinary and symbolic powers of their edge
ideals.

\begin{theorem}\label{unicyclic}
	Let    $D$  be a weighted naturally oriented unicyclic graph with a unique odd cycle $C_{2n+1} = (x_1,\ldots,x_{2n+1})$. Then   $ I(D)^{(s)} = I(D)^s $ for all  $s \geq   2$  if and only if  $w(x) \geq 2$ when $\deg_D(x)\geq 2 $ for all $x \in V(D)$. 
\end{theorem}

\begin{proof}
	If $w(x) \geq 2$ when $\deg_D(x)\geq 2 $ for all $x \in V(D)$, then by Remark \ref{svc1} and Lemma \ref{mandal}, we have ${I(D)}^{(s)}={I(D)}^s $ for all $ s \geq 2.$
	
	Now we assume that ${I(D)}^{(s)}={I(D)}^s $ for all $ s \geq 2.$  First we claim $w(x) \neq 1$  for all  $x \in V(C_{2n+1})$. Suppose its not true.  Without loss of generality we can assume that  $w(x_1)=1.$  Here $N_D^-(x_2) = \{x_1\} \nsubseteq V^+(D)$. Thus $ x_2 \in  V (C_{2n+1}) \setminus N_D^+(V^+(D))$. Hence by Proposition \ref{oddcycle}, $ I(D)^{(n+1)} \neq I(D)^{n+1} $, which is a contradiction. So our claim follows.      
	
	%
	
	Now we claim $w(x) \geq 2$ when $\deg_D(x)\geq 2 $ for all $x \in V(D)\setminus V(C_{2n+1})$. Suppose its not true. Then there exists some  $x_{i_t} \in   V(D)\setminus V(C_{2n+1})      $ such that   $\deg_D(x_{i_t}) \geq 2$ and $w(x_{i_t} )  = 1 $. Without loss of generality we can assume that $x_{i_t} $ is in some tree $T$ connected with $x_1$. As $T$ is a tree, there is only one path $P$ from $x_1 $ to $ x_{i_t} $. Let $P =  x_1x_{i_1} x_{i_2}\ldots x_{i_t}$ be that path whose length is $t$ and since $\deg_D(x_{i_t}) \geq 2$, there exists some $y_j \in N_D^+(x_{i_t})$. Without loss of generality we can assume that $t$ be the least integer such that $\deg_D(x_{i_t}) \geq 2$ and $w(x_{i_t} )  = 1 $. That means  $w(x_{i_1} )  \geq 2,\ldots,w(x_{i_{t-1}} )  \geq 2 $ and $ w(x_{i_{t}} )  = 1 $. Note that $t \geq 1$.
	
\textbf{Case (1)} Suppose $t =1 $.
	
There exists  an induced weighted naturally oriented path $D^{\prime}$  of $D$   with  $V(D^{\prime}) = \{x_{2n+1},x_1,\\x_{i_1},y_j\}$, $E(D^{\prime}) = \{ (x_{2n+1},x_1),(x_{1},x_{i_1}),(x_{i_1},y_j) \}$, $w(x_1) \geq 2$ and $w(x_{i_1}) = 1$.

\textbf{Case (2)} Suppose $t =2 $.

There exists  an induced weighted naturally oriented path $D^{\prime}$  of $D$   with  $V(D^{\prime}) = \{x_1,x_{i_1},\\x_{i_2},y_j\}$, $E(D^{\prime}) = \{ (x_{1},x_{i_1}),(x_{i_1},x_{i_2}),(x_{i_2},y_j) \}$, $w(x_{i_1}) \geq 2$ and $w(x_{i_2}) = 1$.

\textbf{Case (3)} Suppose $t \geq 3 $.

There exists  an induced weighted naturally oriented path $D^{\prime}$  of $D$  with  $V(D^{\prime}) = \{x_{i_{t-2}},x_{i_{t-1}},\\x_{i_t},y_j\}$, $E(D^{\prime}) = \{(x_{i_{t-2}},x_{i_{t-1}}),(x_{i_{t-1}},x_{i_t}   ), (x_{i_t},y_j) \}$, $w(x_{i_{t-1}}) \geq 2$ and $w(x_{i_t}) = 1$.

Thus by Lemma \ref{atmost}, for any $t \geq 1$, we have   $ {I(D)}^{(3)} \neq {I(D)}^3 $, which is a contradiction. Hence the claim follows.
\end{proof}

%
%
%

Now we compare the ordinary and symbolic powers of edge ideals of weighted oriented even cycles.

We observe that   \cite[Lemma 48]{pitones} is true for weighted  oriented even  cycle of length $4$. Hence  we get the following result.
\begin{lemma}\cite[Lemma 48]{pitones}\label{C6}  
	Let $ D $ be  a  weighted  oriented  cycle with underlying graph is $  C_n = (x_1,\ldots,x_n) $ where $ n \geq 4 $ and $ n \neq 5 .$   Then $ D $ has the minimal-strong property if and only if the vertices of $ V^+(D) $ are sinks.  
\end{lemma}

%
%
%


%
%

\begin{theorem}\label{evencycle} 
	  Let $ D $ be a  weighted  oriented even cycle. If   $ V(D) $ is a
	strong vertex cover of $ D $ or
	$ D $ has the minimal-strong property, then $ {I(D)}^{(s)} = {I(D)}^s $ for all $s \geq 2.$      
\end{theorem}
\begin{proof}
If $ V(D) $ is a
strong vertex cover of $ D ,$ then  by Lemma \ref{kanoy},  we get $ {I(D)}^{(s)} = {I(D)}^s $ for all $s \geq 2.$

If 	$ D $ has the minimal-strong property, then by Lemma \ref{C6} and Lemma \ref{DG}, we have $ {I(D)}^{(s)} = {I(D)}^s $ for all $s \geq 2.$    
\end{proof}
\begin{remark}\label{one.weight}
	Converse of the Theorem \ref{evencycle} need not be true.\\	\hspace*{0.5cm}For example consider the  weighted oriented even cycle $D$ as in Figure \ref{fig.1}. Let $I= I(D)$.	
	Then $ I   = ( x_1x_2^{w_2},x_2x_3,x_3x_4,x_4x_1   )$.
	\begin{figure}[!ht]
		\begin{tikzpicture}[scale=0.85]
		\begin{scope}[ thick, every node/.style={sloped,allow upside down}] 
		\definecolor{ultramarine}{rgb}{0.07, 0.04, 0.56} 
		\definecolor{zaffre}{rgb}{0.0, 0.08, 0.66}   
		\draw[fill, black] (0,0) --node {\midarrow}(0,2);
		\draw[fill, black] (0,2) --node {\midarrow}(2,2);
		\draw[fill, black] (2,2) --node {\midarrow}(2,0);
		\draw[fill, black] (2,0) --node {\midarrow}(0,0);
		\draw [fill] [fill] (0,0) circle [radius=0.04];
		\draw[fill] [fill] (2,0) circle [radius=0.04];
		\draw[fill] [fill] (2,2) circle [radius=0.04];
		\draw[fill] [fill] (0,2) circle [radius=0.04];
		
		\node at (-0.4,0) {$x_1$};
		\node at (2.4,0) {$x_4$};
		\node at (2.4,2) {$x_3$};
		\node at (-0.4,1.9) {$x_{2}$}; 
		
		\node at (0,-0.4) {$1$};
		\node at (2,-0.4) {$1$};
		\node at (2,2.4) {$1$};
		\node at (0,2.4) {$w_2 \neq  1$};

		\end{scope}
		\end{tikzpicture}
		\caption{A  weighted naturally oriented even cycle $D$ of length $4$.}\label{fig.1}

	\end{figure}
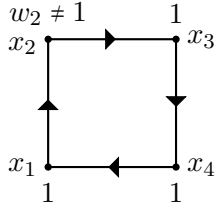	
Since  $N_D^-(x_2)    \cap  V^+(D) = \phi,$  by Remark \ref{s.v.1},  $ V(D) $ is not a
strong vertex cover of $D.$ Let the vertex covers of $D$ except $V(D)$ are $C_1= \{x_1, x_3\}$, $C_2 = \{x_2, x_4\},$ $C_3 = \{x_2, x_3, x_4  \},$ $C_4 =\{x_1, x_2, x_4  \},$ $C_5 =\{x_1, x_3, x_4  \}$ and $C_6 =\{x_1, x_2, x_3  \} .$

Consider  $C_1 =\{ x_1,x_3   \}$.    
Here $ x_2 \in N_D^+(x_1) \cap  C_{1}^c$ and $ x_4 \in N_D^+(x_3) \cap  C_{1}^c$. So $L_1^D(C_1) = \{x_1,x_3\} .$  Note that $C_1$ is minimal and by Lemma \ref{minimal to strong}, $C_1$ is strong.    
Hence $I_{C_{1}} =( x_1,x_3).$        
By the same argument we can prove that  $C_2 $ is strong and   $I_{C_{2}} =(x_2,x_4).$  Consider  $C_3 = \{x_2, x_3, x_4  \}.$
Here $ x_1 \in  N_D^+(x_4) \cap C_{3}^c$,  $ x_3 \notin N_D^+(x_2) \cap C_3^c$ and   $ x_4 \notin N_D^+(x_3) \cap  C_3^c$. So $L_1^D(C_3) = \{x_4\} .$ Here $ x_1 \in N_D^-(x_2) \cap  C_3^c$ and $ x_2 \notin  N_D^-(x_3) \cap   C_3^c.$ Thus $L_2^D(C_3) = \{x_2\}$ and $L_3^D(C_3) = \{x_3\}.$ Since  $ x_2 \in N_D^-(x_3) \cap V^+(D) \cap L_2^D(C_{3}),$    $C_{3}$ is strong. Hence we have $I_{C_{3}} =( x_2^{w_2},x_3^{w_3}, x_{4} )=( x_2^{w_2},x_3, x_{4} )$. Consider  $C_4 = \{x_1, x_2, x_4  \}.$
By Lemma \ref{L3}, $L_3^D(C_4) = \{x_1\}.$ Since  $N_D^-(x_1) = \{x_4\}   \nsubseteq  V^+(D),$ by Remark \ref{s.v.1},  $C_{4}$ is not strong.
By the same argument we can prove that  $C_5 $ and $C_6 $ are not strong.
%
Thus $C_1,C_2$ and $C_3$ are the only strong vertex covers of $D.$ Note that $C_3$ is not a  minimal vertex cover of $ D .$ So $ V(D) $ is neither a
	strong vertex cover of $ D $ nor
	$ D $ has the minimal-strong property.
We claim that    	    
$ I^{(s)} = I^s $ for all $s \geq 2.$

By  Lemma \ref{cooper}, we have     $I^{(s)} =  ((x_2^{w_2},x_3,x_4)\cap(x_2, x_4) )^s \bigcap (x_1, x_3)^s =  (x_2^{w_2},x_2x_3,x_4)^s \cap (x_1, x_3)^s.$ Let $\bar{m} \in \mathcal{G}(I^{(s)} )$. Then $\bar{m} = \lcm(m_1,m_2)$ for some $m_1 \in \mathcal{G}((x_2^{w_2},x_2x_3,x_4)^s  )$ and $m_2 \in \mathcal{G}((x_1, x_3)^s  ).$  
	Thus  $ m_1 =  (x_2^{w_2})^{a_1 } (x_2x_3)^{a_2}(x_4)^{a_3}   $ and  $ m_2 =  (x_1)^{b_1 } (x_3)^{b_2}  $ for some $a_i, b_i \geq 0$ with  $a_1 + a_2 +a_3 = s$ and $b_1 + b_2   = s.$
		
	\textbf{Case (1)} Assume that $b_2 \leq a_2 .$ Then $b_1 \geq a_1 + a_3.$\\
	 Thus $\bar{m} = \lcm(m_1,m_2) = (x_2^{w_2})^{a_1 } (x_2x_3)^{a_2}(x_4)^{a_3} (x_1)^{b_1 }= (x_1x_2^{w_2})^{a_1 } (x_2x_3)^{a_2}(x_4x_1)^{a_3}\\ (x_1)^{b_1 - (a_1 + a_3)}.$
	Hence $\bar{m} \in I^{a_1 +a_2 +a_3} = I^s.$\\
	\textbf{Case (2)} Assume that $b_2 > a_2 .$\\
	Thus $\bar{m} = \lcm(m_1,m_2) = (x_2^{w_2})^{a_1 } (x_2x_3)^{a_2}(x_4)^{a_3}(x_3)^{b_2 -a_2} (x_1)^{b_1 }$.
	Here $b_2 - a_2 + b_1 =  s- a_2 =a_1 + a_3.$ Note that $x_2^{w_2}$ can pair up with $x_1$ to get some element of $\mathcal{G}(I)$. Also $x_2^{w_2}$ can pair up with $x_3$ to get a  multiple of some element of $\mathcal{G}(I)$. Similarly $x_4$ can pair up with $x_3$ or $x_1$ to get some element of   $\mathcal{G}(I)$. Thus   
	$(x_2^{w_2})^{a_1 } (x_4)^{a_3}(x_3)^{b_2 -a_2} (x_1)^{b_1 }$ can be expressed as a multiple of product of $a_1 + a_3$ elements of $\mathcal{G}(I)$.  Therefore $\bar{m} \in I^{ (a_1 + a_3)+a_2} = I^s.$
Hence the claim follows.    
\end{remark}

 In this paper we characterize the weighted naturally oriented even cycles  for the equality of all the ordinary  and symbolic powers of their edge ideals.

\begin{remark}\label{svc}    
	Let $D$ be a weighted oriented  cycle. By Remark \ref{svc1}, $V(D)$ is a strong vertex cover of $D$ if and only if $D$ is naturally oriented and all vertices of $D$  have non-trivial weights.
\end{remark}

By the above remark and Corollary \ref{oddcycle3},
we see that,  
if  $ D $ is a  weighted naturally  oriented  odd   cycle,  then  all vertices of $D$  have non-trivial weights   if and only if    $ I(D)^{(s)} = I(D)^s $ for all  $s \geq   2$.
\begin{lemma}\label{all.weights}
Let $ D $ be a  weighted naturally  oriented   cycle. If all vertices of $D$  have non-trivial weights, then $ {I(D)}^{(s)}={I(D)}^s $ for all $ s \geq 2.$  	
\end{lemma}

\begin{proof}
It follows from  Remark \ref{svc} and Lemma \ref{mandal}.
\end{proof}
%
%
%
%
%


In the next result, if $ D $ is a weighted naturally oriented cycle of length $ n \neq 4 $ where at
least one vertex has non-trivial weight, we see that only the equality of $ 3 $rd ordinary and symbolic
power ensures the non-trivial weight of each vertex.

\begin{theorem}\label{cycle1}
	Let $ D $ be a weighted naturally  oriented   cycle  with underlying graph $G$ is $  C_n=(x_1,x_2,\ldots,x_n) ,$ where $n \neq 4$ and at least one vertex of $D$ has non-trivial weight. Then all vertices of $D$  have non-trivial weights if and only if $ I(D)^{(3)} = I(D)^3 .$

\end{theorem}

\begin{proof}
	
Here $V(D) = \{x_1,\ldots,x_{n} \}$. Let $w_i = w(x_i)$ for $x_i \in V(D).$ If all vertices of $D$  have non-trivial weights, then by Lemma \ref{all.weights}, we have $ {I(D)}^{(3)} = {I(D)}^3.$	

Now we assume that ${I(D)}^{(3)} = {I(D)}^3.$  
Suppose all vertices of $D$ do not  have non-trivial weights. We know that at least one vertex of $D$ has non-trivial weight.
Without loss of generality we can assume that $w(x_2)  \geq 2$ and  $w(x_3)=1.$
	
%
\textbf{Case (1)} Assume that $n \geq 5.$

Then there exists  an induced weighted naturally oriented path $D^{\prime}$ of $D$  with  $V(D^{\prime}) = \{x_{1},x_{2},x_{3},x_{4}\}$, $E(D^{\prime}) = \{(x_{1},x_{2}),(x_{2},x_{3}), (x_{3},x_{4}) \}$, $w(x_{2}) \geq 2$ and $w(x_{3}) = 1$. By Lemma \ref{atmost},  we have   $x_1x_2^{w_2}x_3^2x_4^{w_4} \in  {I(D)}^{(3)} \setminus {I(D)}^3 $, which is a contradiction.

\textbf{Case (2)} Assume that $n =3.$

Note that  $ x_1 \in  V (C_{3}) \setminus N_D^+(V^+(D))$. By 
Proposition \ref{oddcycle},   $x_1^{w_1}x_2^{w_2}x_3 \in {I(D)}^{(2)} \setminus {I(D)}^2$. Then by Lemma \ref{edge} and Lemma \ref{cooper},  we have $(x_1^{w_1}x_2^{w_2}x_3)(x_1x_2^{w_2}) \in {I(D)}^{(3)} \setminus {I(D)}^3$, which is a contradiction.    
\end{proof}

\begin{remark}\label{C6.}
	Let $ D $ be a weighted naturally oriented cycle of length $n$  with underlying graph $G$ is $  C_n=(x_1,x_2,\ldots,x_n) ,$ where $ n \neq 4,6 $ and at least one vertex of $D$ has non-trivial weight. If we assume $w(x_2)  \geq 2$ and  $w(x_3)=1$,
	 then by the similar argument as in Theorem \ref{cycle1}, we find that $x_1x_2^{w_2}x_3x_6^{w_6} \in I(D)^{(2)} \setminus I(D)^2$  for $ n \geq  7 $ and $x_1^{w_1}x_2^{w_2}x_3 \in I(D)^{(2)} \setminus I(D)^2$  for $ n = 3 $
	and $ 5 $. Hence $I(D)^{(2)} = I(D)^2$ implies all vertices of $ D $ have non-trivial weights and it ensures the equality of all the
 ordinary and symbolic powers. But it is not true for weighted naturally oriented
	even cycles of length $ 6. $ For example consider $ D $ to be a weighted naturally oriented cycle $ D $ where the
	underlying graph is $  C_6 = (x_1, x_2, \ldots, x_6) $ and only  $ x_2 $ has  non-trivial
	weight $2$. Then $ I(D) = (x_1x_2^{2}, x_2x_3, x_3x_4, x_4x_5, x_5x_6, x_6x_1).$ Using
	Macaulay 2, we observe $I(D)^{(2)} = I(D)^2$ but $I(D)^{(3)} \neq I(D)^3$ and in this case all the
	vertices except $ x_2 $ have trivial weights. In Theorem \ref{cycle1}, $I(D)^{(3)} = I(D)^3$ guarantees that all vertices of $ D $ have non-trivial weights and it ensures the equality of all the
 ordinary and symbolic powers. But if $ D $ is a weighted naturally oriented cycle  with
	underlying graph is  $ C_4 = (x_1, x_2, x_3, x_4) $ and at least one vertex of $ D $ has non-trivial
	weight,  we do not even need  each vertex to be of
	non-trivial weight to ensure the equality of all the
 ordinary and symbolic powers (see Proposition \ref{evencycle2}).            	
\end{remark}

In the next two results, we give necessary and sufficient condition for the equality of  ordinary and symbolic powers of   edge ideals of  weighted naturally oriented even cycles.

\begin{corollary}\label{evencycle1}
	Let $ D $ be a weighted naturally oriented even cycle   with underlying graph  $C_{n} = (x_1,\ldots,x_{n}),$  where $n \neq 4$ and at least one vertex of $D$ has non-trivial weight. Then   $ I(D)^{(s)} = I(D)^s $ for all  $s \geq   2$  if and only if  all vertices of $D$  have non-trivial weights.

\end{corollary} 

\begin{proof}
If $ I(D)^{(s)} = I(D)^s $ for all  $s \geq   2$,  then by Theorem \ref{cycle1},  all vertices of $D$  have non-trivial weights.	
If  all vertices of $D$  have non-trivial weights,	
  then by Lemma \ref{all.weights},  we have $ I(D)^{(s)} = I(D)^s $ for all  $s \geq   2$ 
%
\end{proof}

In the next result, we characterize the  weighted naturally oriented even cycles of length $ 4 $ for the equality of ordinary and symbolic powers of their edge
ideals.
	\begin{proposition}\label{evencycle2}
	Let $ D $ be a weighted naturally oriented even cycle  with underlying graph  $C_{4} = (x_1,x_2,x_3,x_{4})$  and at least one vertex of $D$ has non-trivial weight. Then $ I(D)^{(s)} = I(D)^s $ for all $s \geq 2$   if and only if $D$ satisfies one of the following conditions:
	\begin{enumerate}
		\item	all vertices of $D$  have non-trivial weights,
		  	  
		\item one vertex of $D$  has non-trivial weight,

		\item only two non-consecutive vertices of $D$  have non-trivial weights.

	\end{enumerate}

\end{proposition}

\begin{proof}
Let $I = I(D)$.	If  $D$  satisfies $ (1), $ then by Lemma \ref{all.weights}, we have  $ I^{(s)} = I^s $ for all  $s \geq   2$.
	If  $D$  satisfies $ (2), $ then by Remark \ref{one.weight}, we get  $ I^{(s)} = I^s $ for all  $s \geq   2$. Now assume  $D$  satisfies $ (3). $ Here two non-consecutive vertices of $D$  have non-trivial weights. Without loss of generality we can assume that $w(x_2) \neq 1$ and $w(x_4) \neq 1.$ Then $I=(x_1x_2^{w_2},x_2x_3,x_3x_4^{w_4},x_4x_1)$ and by the similar argument as in Remark \ref{one.weight}, we find	
$I^{(s)} = (x_1,x_3 )^s \bigcap  ((x_2^{w_2},x_3,x_4)\cap(x_2, x_4) )^s \bigcap ((x_1,x_2,x_4^{w_4})\cap(x_2, x_4) )^s = (x_1,x_3 )^s \cap  (x_2^{w_2},x_2x_3,x_4)^s \cap (x_2, x_1x_4, x_4^{w_4})^s.$  We claim  that $I^{(s)}   \subseteq I^s.$\\
	We prove this by induction on $s.$ The case for $s=1$ is trivial. Let $\bar{m} \in \mathcal{G}(I^{(s)} )$. Then $\bar{m} = \lcm(m_1,m_2,m_3)$ for some $m_1 \in \mathcal{G}((x_1,x_3)^s  )$, $m_2 \in \mathcal{G}((x_2^{w_2},x_2x_3,x_4)^s  )$ and $m_3 \in \mathcal{G}((x_2, x_1x_4, x_4^{w_4})^s  ).$ Thus  $ m_1 =  (x_1)^{a_1 } (x_3)^{a_2}   $, $ m_2 =  (x_2^{w_2})^{b_1 } (x_2x_3)^{b_2}(x_4)^{b_3}   $ and  $ m_3 =  (x_2)^{c_1 } (x_1x_4)^{c_2}(x_4^{w_4})^{c_3}  $ for some $a_i, b_i,c_i \geq 0$ with  $a_1 + a_2  = s,$ $b_1 + b_2 + b_3 = s$ and    $c_1 + c_2 +c_3  = s.$  
	\vspace*{0.1cm}\\	
	\textbf{Case (1)} Assume that $a_1 \neq 0.$
	
If $b_3 \neq 0,$ then $\bar{m}$ is divisible by $x_1x_4$ and observe that $\frac{\bar{m}}{x_1x_4}$ $ \in  (x_1,x_3 )^{s-1} \cap  (x_2^{w_2},x_2x_3,x_4)^{s-1}\\ \cap (x_2, x_1x_4, x_4^{w_4})^{s-1}=I^{(s-1)}.$ Hence by induction hypothesis	$\frac{\bar{m}}{x_1x_4}$ $ \in  I^{s-1}$ and so $\bar{m} \in I^s.$

If $b_2 \neq 0,$ then $\bar{m}$ is divisible by $x_2x_3$ and notice that $\frac{\bar{m}}{x_2x_3}$ $ \in  (x_1,x_3 )^{s-1} \cap  (x_2^{w_2},x_2x_3,x_4)^{s-1} \cap (x_2, x_1x_4, x_4^{w_4})^{s-1}=I^{(s-1)}.$ Hence by induction hypothesis	$\frac{\bar{m}}{x_2x_3}$ $ \in  I^{s-1}$ and so $\bar{m} \in I^s.$

Now we assume $b_3=b_2=0.$ Then $b_1=s$ and  $\lcm(m_1,m_2) \in I^s.$ So $\bar{m} \in I^s.$ 
	
\textbf{Case (2)} Assume that $a_1 = 0.$
Then $a_2=s$ and  $\lcm(m_1,m_3) \in I^s.$ Hence $\bar{m} \in I^s.$
\vspace*{0.1cm}\\
Next we prove the converse part. Let us assume that $ I^{(s)} = I^s $ for all  $s \geq   2$. Suppose none of $ (1), $ $ (2)  $ and $(3)$ is true. Then $D$ must satisfy one of the following conditions:

\item [(a)]	only two consecutive vertices of $D$  have non-trivial weights,
	
\item [(b)] only three  vertices of $D$  have non-trivial weights.
	
\item [(a)]	Assume that $D$ has only two consecutive vertices with non-trivial weights. Without loss of generality we can assume that $w(x_2) \neq 1$ and $w(x_3) \neq 1.$ Then $I=(x_1x_2^{w_2},x_2x_3^{w_3},\\x_3x_4,x_4x_1)$ and by the similar argument as in Remark \ref{one.weight}, we find that	
$I^{(3)} =  ((x_1,x_3^{w_3},x_4)\\\cap(x_1, x_3) )^3  \bigcap ((x_2^{w_2},x_3^{w_3},x_4)\cap(x_2, x_4) )^3 =  (x_1,x_3^{w_3},x_3x_4)^3 \cap (x_2^{w_2},x_2x_3^{w_3},x_4)^3.$ Observe that $x_1x_2x_3^{w_3}x_4^2 \in I^{(3)} \setminus I^3$, which is a contradiction.  

\item [(b)]	Assume that $D$ has only three  vertices with non-trivial weights. Without loss of generality we can assume that $w(x_2) \neq 1,$ $w(x_3) \neq 1$ and $w(x_4) \neq 1.$ Then $I=(x_1x_2^{w_2},x_2x_3^{w_3},x_3x_4^{w_4},\\x_4x_1)$ and by the similar argument as in Remark \ref{one.weight}, we find that	
$I^{(2)} =  ((x_1,x_3^{w_3},x_4^{w_4})\cap(x_1, x_3) )^2 \bigcap ((x_1,x_2,x_4^{w_4})\cap(x_2, x_4) )^2 \bigcap ((x_2^{w_2},x_3^{w_3},x_4)\cap(x_2, x_4) )^2 =  (x_1,x_3^{w_3},x_3x_4^{w_4})^2 \cap (x_2,x_1x_4,x_4^{w_4})^2 \cap  (x_2^{w_2},x_2x_3^{w_3},x_4)^2.$ Notice that $x_1x_2x_3x_4^{w_4} \in I^{(2)} \setminus I^2$, which is a contradiction. Hence the proof follows.    
\end{proof}

\section{Comparing symbolic powers  of weighted oriented graphs}


In this section, we show that the symbolic powers of edge ideals of a weighted oriented graph $ D $ and the new weighted oriented graph $ D^{\prime} $ obtained from $D$ after replacing the weights of vertices with non-trivial weights which are sink, by trivial
weights, behave in a similar way. We see that using the symbolic powers of edge ideals of one class of weighted oriented graphs, we can compute the symbolic powers of edge ideals of another class of weighted oriented graphs.

\begin{notation} \label{phi.}    
	Let $ D $ be a  weighted oriented graph,  where  $ U \subseteq V^{+}(D)  $ be the set of vertices  which are sinks and $ w_j =w(x_j) $  if $ x_j \in V^{+}(D) .$ Let $ D^{\prime} $ be the  weighted oriented graph obtained from  $D$ after replacing  $ w_j $ by $ w_j = 1  $  if $ x_j \in U .$ 	Let $ V(D) = V(D^{\prime}) = V = \{x_1,\ldots,x_n\}. $
	Let $R=k[x_1,\ldots,x_n]=$  $\displaystyle{{\bigoplus_{d=0}^{\infty}}R_d}$ be the standard graded polynomial ring. Consider the map
	\begin{align*}
	\Phi : R \longrightarrow R ~ \mbox{where}  ~ x_j \longrightarrow x_j ~ \mbox{if} ~ x_j \notin U  \mbox{and} ~ x_j \longrightarrow x_j^{w_j}~  \mbox{if}  ~ x_j \in U .
	\end{align*}
\end{notation}

Here $ \Phi $ is an injective homomorphism of $ k- $algebras.
By \cite[Corollary 5]{gimenez}, $I(D)$ is Cohen–
Macaulay if and only if $I(D^{\prime})$ is Cohen–Macaulay.
We want to investigate the relationship between the symbolic powers of $I(D)$ and $I(D^{\prime}).$
The next two lemma's are very important to prove our result in Theorem \ref{sym.theorem.1}.    

\begin{lemma}\label{strong}
	Let $D,$ $D^{\prime}$ and $\Phi$ are same as defined in Notation \ref{phi.}. Then $C$ is a strong vertex cover of $D$  if and only if $C$ is a strong vertex cover of $D^{\prime}$.  
\end{lemma}

\begin{proof}
	Since $D$ and $D^{\prime}$ have the same underlying graph, 
	$C$ is a vertex cover of $D$ 	if and only if $C$ is a vertex cover of $D^{\prime}$. Now consider $C$ to be a  vertex cover of both $D$ and $D^{\prime}.$
	Here $D$ and $D^{\prime}$ have the same orientation on edges. Thus 
	$L_i^D(C)$	$=$ $L_i^{D^{\prime}}(C)$  for $1 \leq i \leq 3$. This implies $C \setminus L_1^D(C)$  	$=$ $C \setminus L_1^{D^{\prime}}(C)$.
	Since each element of $U$ is sink, $N_{D}^-(x) \cap U =  \phi$ for $x \in V \setminus U.$ Since two adjacent vertices can not be sink vertices, $N_{D}^-(x) \cap U =  \phi$ for $x \in  U.$  Hence for each $ x \in C  $,    $N_{D}^-(x) \cap U =  \phi$.        
Note that $V^+(D)\setminus  U =     V^+(D^{\prime}).$ Since $D$ and $D^{\prime}$ have the same orientation on edges, $N_{D}^-(x)  =  N_{D^{\prime}}^-(x)$ for each $x \in C.$
	Thus for each $ x \in C  $, $N_{D}^-(x) \cap V^+(D) =N_{D}^-(x) \cap [V^+(D)\setminus  U] =  N_{D^{\prime}}^-(x) \cap [V^+(D)\setminus  U] =  N_{D^{\prime}}^-(x) \cap  V^+(D^{\prime}).$
Therefore for each $ x \in $ $L_3^D(C) =$  $L_3^{D^{\prime}}(C),$  we have  $N_{D}^-(x) \cap V^+(D)\cap [C\setminus {L_1^D{(C)}} ] =  N_{D^{\prime}}^-(x) \cap  V^+(D^{\prime})\cap [C\setminus {L_1^{D^{\prime}}{(C)}} ]$ and  $N_{D}^-(x) \cap V^+(D)\cap [C\setminus {L_1^{D}{(C)}} ] \neq \phi $ if and only if $  N_{D^{\prime}}^-(x) \cap  V^+(D^{\prime})\cap [C\setminus {L_1^{D^{\prime}}{(C)}} ]\neq \phi$.  
	Hence by Remark \ref{s.v.1}, $C$ is a strong vertex cover of $D$ 	if and only if $C$ is a strong vertex cover of $D^{\prime}$.  
\end{proof}

\begin{lemma}\label{intersections}
	Let $D,$ $D^{\prime}$ and $\Phi$ are same as defined in Notation \ref{phi.}. 
	Let   ${I}$ and $\tilde{I}$ be the edge ideals of   $D$ and $D^{\prime},$  respectively.  Let $C_{1_1},\ldots,C_{{r}_1}$ are the maximal strong vertex covers of both $ D $ and $ D^{\prime}.$         
	Let $C_{i_2},\ldots,C_{i_{t_i}}$ are the strong vertex covers of both $ D $ and $ D^{\prime} $ such that $C_{i_j}  \subset C_{i_1} $ for $2 \leq j \leq t_i$  and $1 \leq i \leq r.$ Let  ${I}_{C_{i_j}}$  and  $\tilde{I}_{C_{k_l}}     $ are the irreducible ideals associated to $ C_{i_j} $ and $C_{k_l} ,$ respectively. Then
	$$\Phi(\tilde{I}_{C_{i_1}}  \cap \tilde{I}_{C_{i_2}} \cap   \cdots \cap  \tilde{I}_{C_{i_{t_i}}}    ) = {I}_{C_{i_1}}  \cap {I}_{C_{i_2}} \cap   \cdots \cap  {I}_{C_{i_{t_i}}}   ~\mbox{for} ~ 1 \leq i \leq r.$$
	
	Moreover, every element of $ \mathcal{G}( ({I}_{C_{1_1}}  \cap {I}_{C_{1_2}} \cap   \cdots \cap  {I}_{C_{1_{t_1}}})^s       )$ is of the form $$\displaystyle{ (\prod_{x_j} x_j^{d_{j_1}} | x_j \in V \setminus U)(\prod_{x_k} (x_k^{w_k})^{e_{k_1}} | x_k \in  U)} \mbox{~for some    $d_{j_1},e_{k_1} \geq  0.$}$$
\end{lemma}

\begin{proof}
	First we claim that  $       \Phi(\tilde{I}_{C_{1_1}}  \cap \tilde{I}_{C_{1_2}} \cap   \cdots \cap  \tilde{I}_{C_{1_{t_1}}}) =   {I}_{C_{1_1}}  \cap {I}_{C_{1_2}} \cap   \cdots \cap  {I}_{C_{1_{t_1}}}.    $	
	Consider any strong  vertex cover $C_{1_i}$ of both $D$ and $D^{\prime}.$ From the proof of Lemma \ref{strong},
  $L_p^D(C_{1_i})$	$=$ $L_p^{D^{\prime}}(C_{1_i})$  for $1 \leq p \leq 3$. Notice that  $N_{D}^+(x) = N_{D^{\prime}}^+(x)  = \phi $ for each $x \in U$. So   $L_1^{D}(C_{1_i}) \cap U = L_1^{D^{\prime}}(C_{1_i}) \cap U = \phi$. This implies  $ U \subseteq  L_2^{D}(C_{1_i}) \cup L_3^{D}(C_{1_i}) = L_2^{D^\prime}(C_{1_i}) \cup L_3^{D^\prime}(C_{1_i})  $. By definition   
	$ \tilde{I}_{C_{1_i}} = ( L_1^{D^\prime}(C_{1_i}) \cup \{x_j^{w_j}~|~x_j \in [L_2^{D^\prime}(C_{1_i}) \cup L_3^{D^\prime}(C_{1_i})] \setminus U )\} \cup \{x_k~|~x_k \in  U )\} ) $ and $ {I}_{C_{1_i}} = ( L_1^{D}(C_{1_i}) \cup \{x_j^{w_j}~|~x_j \in [L_2^{D}(C_{1_i}) \cup L_3^{D} (C_{1_i})] \setminus U )\} \cup \{x_k^{w_k}~|~x_k \in  U )\} ) $.	 
	Note that $ \Phi(\tilde{I}_{C_{1_i}})  = {I}_{C_{1_i}} $ for $1 \leq i \leq t_1.$	 By Lemma \ref{atiyah},  we have $\Phi(\tilde{I}_{C_{1_1}}  \cap \tilde{I}_{C_{1_2}} \cap   \cdots \cap  \tilde{I}_{C_{1_{t_1}}}    ) \subseteq \Phi(\tilde{I}_{C_{1_1}})  \cap \Phi(\tilde{I}_{C_{1_2}}) \cap   \cdots \cap  \Phi(\tilde{I}_{C_{1_{t_1}}}    ) = {I}_{C_{1_1}}  \cap {I}_{C_{1_2}} \cap   \cdots \cap  {I}_{C_{1_{t_1}}} $.
\vspace*{0.2cm}\\	
	Let $ f \in \mathcal{G}({I}_{C_{1_1}}  \cap {I}_{C_{1_2}} \cap   \cdots \cap  {I}_{C_{1_{t_1}}}).$ Then 
	$f= \lcm (f_1,f_2,\ldots,f_{t_1})$ for some  $f_i \in \mathcal{G}({I}_{C_{1_i}}  )$ where $1 \leq i \leq t_1.$
	Here  each $ f_i $ involves only one variable. Fix any $i \in [t_1]$. If $f_i$ involves the variable $x_l,$ then 
	 \begin{equation*}  
	f_i= 
	\begin{cases}
	&x_l 
	\vspace*{0.2cm}~\mbox{if}~  x_l \in  L_1^{D}(C_{1_i}) \\
	&x_l^{w_l}   ~\mbox{if}~  x_l \in [L_2^{D}(C_{1_i}) \cup L_3^{D} (C_{1_i})] \setminus U\vspace*{0.2cm}\\
	&x_l^{w_l}  ~\mbox{if}~  x_l \in  U       
	\end{cases}
	\end{equation*}

Thus we can express  $f$ as
\begin{equation}\label{f}
f = \displaystyle{(\prod_{x_j} (x_j^{a_j})^{b_{j}} | x_j \in V \setminus U)(\prod_{x_k} (x_k^{w_k})^{c_{k}} | x_k \in  U)}
\end{equation}	  
 for some  $a_{j} = 1 $ or $w_{j},$ $b_{j} = 0 $ or $1$ and $c_k = 0 $ or $1.$ We want to find some $g \in \tilde{I}_{C_{1_1}}  \cap \tilde{I}_{C_{1_2}} \cap   \cdots \cap  \tilde{I}_{C_{1_{t_1}}} $ such that $\Phi(g) = f.$ We set $g_i=f_i$ if 
  $f_i$ involves variable from $V\setminus  U$ and $g_i=x_l$ if  
  $f_i = x_l^{w_l}$ where  $x_l \in   U$ for $1 \leq i \leq t_1.$ 
Therefore
\begin{equation*}  
g_i= 
\begin{cases}
&x_l 
\vspace*{0.2cm}~\mbox{if}~  x_l \in  L_1^{D^\prime}(C_{1_i}) \\
&x_l^{w_l}   ~\mbox{if}~  x_l \in [L_2^{D^\prime}(C_{1_i}) \cup L_3^{D^\prime} (C_{1_i})] \setminus U\vspace*{0.2cm}\\
&x_l  ~\mbox{if}~  x_l \in  U       
\end{cases}
\end{equation*}	  
	  
Here each $g_i \in \mathcal{G}({\tilde{I}}_{C_{1_i}}  )$ and	$\displaystyle{ \lcm (g_1,g_2,\ldots,g_{t_1}) = (\prod_{x_j} (x_j^{a_j})^{b_{j}} | x_j \in V \setminus U)(\prod_{x_k} (x_k)^{c_{k}} }| x_k \in  U).$  Let $g=\lcm (g_1,g_2,\ldots,g_{t_1}).$ Notice that $\Phi(g) = f.$   Thus $f = \Phi(g)= \Phi(\lcm (g_1,g_2,\ldots,g_{t_1}))$$ \in \Phi(\tilde{I}_{C_{1_1}}  \cap \tilde{I}_{C_{1_2}} \cap   \cdots \cap  \tilde{I}_{C_{1_{t_1}}}    ) .$ Hence $       \Phi(\tilde{I}_{C_{1_1}}  \cap \tilde{I}_{C_{1_2}} \cap   \cdots \cap  \tilde{I}_{C_{1_{t_1}}}    ) =   {I}_{C_{1_1}}  \cap {I}_{C_{1_2}} \cap   \cdots \cap  {I}_{C_{1_{t_1}}}    $.
	By the similar argument, for $2 \leq i \leq r,$  we can prove that
$$       \Phi(\tilde{I}_{C_{i_1}}  \cap \tilde{I}_{C_{i_2}} \cap   \cdots \cap  \tilde{I}_{C_{i_{t_i}}}    ) = {I}_{C_{i_1}}  \cap {I}_{C_{i_2}} \cap   \cdots \cap  {I}_{C_{i_{t_i}}}.          $$

Let $h \in \mathcal{G}( ({I}_{C_{1_1}}  \cap {I}_{C_{1_2}} \cap   \cdots \cap  {I}_{C_{1_{t_1}}})^s       ).$ Then $h = h_1\cdots h_s$ for some $h_i$'s  $ \in   \mathcal{G}({I}_{C_{1_1}}  \cap {I}_{C_{1_2}} \cap   \cdots \cap  {I}_{C_{1_{t_1}}}).$
By (\ref{f}),	for $1 \leq i \leq s,$ $\displaystyle{h_i = (\prod_{x_j} (x_j^{a_{j_i}})^{b_{j_i}} | x_j \in V \setminus U)(\prod_{x_k} (x_k^{w_k})^{c_{k_i}} | x_k \in  U)}$\\
		\hspace*{6cm}    for some  $a_{j_i} = 1 $ or $w_{j},$ $b_{j_i} = 0 $ or $1$ and  $c_{k_i} = 0 $ or $1.$

Hence $\displaystyle{h = (\prod_{x_j} x_j^{d_{j_1}} | x_j \in V \setminus U)(\prod_{x_k} (x_k^{w_k})^{e_{k_1}} | x_k \in  U)}$ where $d_{j_1} = {a_{j_1}}{b_{j_1}} + \cdots + {a_{j_s}}{b_{j_s}} $ and $e_{k_1} = {c_{k_1}} + \cdots + {c_{k_s}}.$  
\end{proof}

In the following theorem, we show that the symbolic powers of edge ideals of $D$ and $D^{\prime}$ behave in a similar way.


\begin{theorem}\label{sym.theorem.1}
	Let   ${I}$ and $\tilde{I}$ be the edge ideals of   $D$ and $D^{\prime},$  respectively. Then $ \Phi(\tilde{I}^s)  = {I}^s   $ and $  \Phi(\tilde{I}^{(s)})  = {I}^{(s)}$ for all $ s \geq 1.$
	
\end{theorem}        
\begin{proof}

By the definitions of ${I}$ and $\tilde{I}$, $ \Phi(\tilde{I})  =  {I}   $. Thus by Lemma \ref{atiyah}, we have $ \Phi(\tilde{I}^s)  = (\Phi(\tilde{I}) )^s =  {I}^s   $ for all $s \geq 1.$ Now we  claim that $  \Phi(\tilde{I}^{(s)})  = {I}^{(s)}.$ Let $C_{1_1},\ldots,C_{{r}_1}$ are the maximal strong vertex covers of both $ D $ and $ D^{\prime} $.       
	Let $C_{i_2},\ldots,C_{i_{t_i}}$ are the strong vertex covers of both $ D $ and $ D^{\prime} $ such that $C_{i_j}  \subset C_{i_1} $ for $2 \leq j \leq t_i$  and $1 \leq i \leq r.$ Then

$\Phi(\tilde{I}^{(s)}) =      \Phi((\tilde{I}_{C_{1_1}}  \cap \tilde{I}_{C_{1_2}} \cap   \cdots \cap  \tilde{I}_{C_{1_{t_1}}}    )^s \cap \cdots \cap (\tilde{I}_{C_{r_1}}  \cap \tilde{I}_{C_{r_2}} \cap   \cdots \cap  \tilde{I}_{C_{r_{t_r}}}    )^s)\hspace*{0.5cm}$ (by Lemma \ref{cooper})\\ $\hspace*{1.35cm}\subseteq   \Phi((\tilde{I}_{C_{1_1}}  \cap \tilde{I}_{C_{1_2}} \cap   \cdots \cap  \tilde{I}_{C_{1_{t_1}}}    )^s) \cap \cdots \cap \Phi((\tilde{I}_{C_{r_1}}  \cap \tilde{I}_{C_{r_2}} \cap   \cdots \cap  \tilde{I}_{C_{r_{t_r}}}    )^s)\\\hspace*{11.9cm}$ (by Lemma \ref{atiyah})\\
	$\hspace*{1.35cm}=   ({I}_{C_{1_1}}  \cap {I}_{C_{1_2}} \cap   \cdots \cap  {I}_{C_{1_{t_1}}} )^s  \cap \cdots \cap  ({I}_{C_{r_1}}  \cap {I}_{C_{r_2}} \cap   \cdots \cap  {I}_{C_{r_{t_r}}} )^s$\\\hspace*{9.3cm} (by Lemma \ref{atiyah} and Lemma \ref{intersections})\\
$\hspace*{1.35cm}=  {I}^{(s)} $	\hspace*{6.85cm} (by Lemma \ref{cooper}).        
	
	Hence $  \Phi(\tilde{I}^{(s)})  \subseteq {I}^{(s)}$ for all $ s \geq 1. $ 	It remains to show that $  {I}^{(s)}  \subseteq \Phi(\tilde{I}^{(s)})$ for all $ s \geq 1. $

	Let $ q \in \mathcal{G}({I}^{(s)}) = \mathcal{G}(({I}_{C_{1_1}}  \cap {I}_{C_{1_2}} \cap   \cdots \cap  {I}_{C_{1_{t_1}}} )^s  \cap \cdots \cap  ({I}_{C_{r_1}}  \cap {I}_{C_{r_2}} \cap   \cdots \cap  {I}_{C_{r_{t_r}}} )^s ) $.
	Then $q = \lcm(q_1,q_2,\ldots,q_r)$ for some $q_i \in \mathcal{G}( ({I}_{C_{i_1}}  \cap {I}_{C_{i_2}} \cap   \cdots \cap  {I}_{C_{i_{t_i}}})^s)$ where $1 \leq i \leq r.$ By Lemma \ref{intersections}, for $1 \leq i \leq r,$ we have\\
	$\hspace*{1cm}q_i = \displaystyle{(\prod_{x_j} x_j^{d_{j_i}} | x_j \in V \setminus U)(\prod_{x_k} (x_k^{w_k})^{e_{k_i}} | x_k \in  U)}$ for some    $d_{j_i},e_{k_i} \geq  0 .$

	
	Hence $\displaystyle{ q = \lcm(q_1,q_2,\ldots,q_r) =    (\prod_{x_j}    x_j^{u_{j}} | x_j \in V \setminus U)(\prod_{x_k} (x_k^{w_k})^{v_{k}} | x_k \in  U) }$ where each $u_j = \max\\ \{d_{j_1},\ldots,d_{j_r}\}$ and  $v_k = \max \{e_{k_1},\ldots,e_{k_r}\}$.
	Let $p_i = \displaystyle{(\prod_{x_j} x_j^{d_{j_i}} | x_j \in V \setminus U)(\prod_{x_k} (x_k)^{e_{k_i}} | x_k \in  U)}$ for $1 \leq i \leq r.$  
Then $  \lcm(p_1,p_2,\ldots,p_r) = \displaystyle{   (\prod_{x_j} x_j^{u_{j}} | x_j \in V \setminus U)(\prod_{x_k} (x_k)^{v_{k}} | x_k \in  U) }$   because each $u_j = \max \{d_{j_1},\ldots,d_{j_r}\}$ and  $v_k = \max \{e_{k_1},\ldots,e_{k_r}\}$. 
	 Since $\Phi(p_i) = q_i\in  ({I}_{C_{i_1}}  \cap {I}_{C_{i_2}} \cap   \cdots \cap  {I}_{C_{i_{t_i}}})^s=\Phi((\tilde{I}_{C_{i_1}}  \cap \tilde{I}_{C_{i_2}} \cap   \cdots \cap  \tilde{I}_{C_{i_{t_i}}}    )^s)  $ and $ \Phi $ is an injective, $p_i \in (\tilde{I}_{C_{i_1}}  \cap \tilde{I}_{C_{i_2}} \cap   \cdots \cap  \tilde{I}_{C_{i_{t_i}}}    )^s$ for $1 \leq i \leq r.$ Let $p = \lcm(p_1,p_2,\ldots,p_r).$  Then $ \Phi(p) = q $ where $p =   \lcm(p_1,p_2,\ldots,p_r) \in (\tilde{I}_{C_{1_1}}  \cap \tilde{I}_{C_{1_2}} \cap   \cdots \cap  \tilde{I}_{C_{1_{t_1}}}    )^s \cap \cdots \cap (\tilde{I}_{C_{r_1}}  \cap \tilde{I}_{C_{r_2}} \cap   \cdots \cap  \tilde{I}_{C_{r_{t_r}}}    )^s =\   \tilde{I}^{(s)}$.
Thus $q \in  \Phi(\tilde{I}^{(s)}) $.
	Hence $   {I}^{(s)} = \Phi(\tilde{I}^{(s)})  $ for all $ s \geq 1.$                	
\end{proof}

In the next result, we prove that the symbolic defects of edge ideals of $D$ and $D^{\prime}$ are same. 
\begin{proposition}\label{sdefect}
	Let $D$, $D^{\prime}$ and $\Phi$   are same as defined in Notation \ref{phi.}. Let   ${I}$ and $\tilde{I}$ be the edge ideals of   $D$ and $D^{\prime},$  respectively. Then  for each $s \geq 1,$ $$\sdefect(\tilde{I},s) = \sdefect(I,s).$$     
\end{proposition}

\begin{proof}
Fix any $s \geq 1.$ Let $X = \mathcal{G}({I}^{(s)}) \setminus {I}^s$ and $X^{\prime} = \mathcal{G}(\tilde{I}^{(s)}) \setminus \tilde{I}^s$.
First we claim that $\Phi(  X^{\prime} ) \subseteq X.$

Suppose $ p $ 
 $\in  \mathcal{G}(\tilde{I}^{(s)}) \setminus \tilde{I}^s $.
By the similar argument  used to get the general form of any element of $ \mathcal{G}({I}^{(s)})  $ in Theorem \ref{sym.theorem.1}, we can write
$\displaystyle{ p =    (\prod_{x_j}    x_j^{u_{j}} | x_j \in V \setminus U)(\prod_{x_k} (x_k)^{v_{k}} | x_k \in  U) }$
for   some $ u_j , v_k \geq    0. $ By Theorem \ref{sym.theorem.1}, $\displaystyle{ \Phi(p) =    (\prod_{x_j}    x_j^{u_{j}} | x_j \in V \setminus U)(\prod_{x_k} (x_k^{w_k})^{v_{k}} | x_k \in  U) } \in {I}^{(s)}$. Let $q = \Phi(p).$ Suppose $q \notin $  $ \mathcal{G}({I}^{(s)})  $. Then $q$ must be multiple of some element of $ \mathcal{G}({{I}}^{(s)})$ (say $q^{\prime}$). By Theorem \ref{sym.theorem.1}, we can write $ q^{\prime}=$$\displaystyle{      (\prod_{x_j}    x_j^{{u_{j}}^{\prime}} | x_j \in V \setminus U)(\prod_{x_k} (x_k^{w_k})^{{v_{k}}^{\prime}} | x_k \in  U) }$  where  each $   {{u_{j}}^{\prime}} \leq u_j ,    {{v_{k}}^{\prime}} \leq  v_k $ and at least one $   {{u_{j}}^{\prime}} < u_j $ or $   {{v_{k}}^{\prime}} <  v_k $.
Let $\displaystyle{ p^{\prime} =     (\prod_{x_j}    x_j^{{u_{j}}^{\prime}} | x_j \in V \setminus U)}\\{(\prod_{x_k} (x_k)^{{v_{k}}^{\prime}} | x_k \in  U) }$. By Theorem \ref{sym.theorem.1}, we have ${I}^{(s)} = \Phi(\tilde{I}^{(s)})$. Since $ \Phi(p^{\prime}) = q^{\prime} \in {I}^{(s)} = \Phi(\tilde{I}^{(s)})$ and $ \Phi $ is an injective, $p^{\prime} \in   \tilde{I}^{(s)}$. Thus $ p $ 
 is  multiple of $p^{\prime}\in  \tilde{I}^{(s)},$ which is a contradiction  because $p \in  \mathcal{G}(\tilde{I}^{(s)})  $. Therefore $q \in \mathcal{G}({I}^{(s)})$. Since $ p $ 
  $\notin   \tilde{I}^s ,$ it is easy to see that $ q $ 
  $\notin   {I}^s$. So $ q $ 
  $\in  \mathcal{G}({I}^{(s)}) \setminus   {I}^s $ and hence $\Phi(  X^{\prime} ) \subseteq X.$

Now consider the map $	\Phi|_{X^{\prime}} : X^{\prime} \longrightarrow X. $ It is enough to show $\Phi|_{X^{\prime}}$ is bijective.
We know $\Phi|_{X^{\prime}}$ is injective.  
Suppose $ g \in  \mathcal{G}({I}^{(s)}) \setminus {I}^s $. Then by   Theorem \ref{sym.theorem.1}, there exists $ f \in  \mathcal{G}(\tilde{I}^{(s)}) \setminus \tilde{I}^s $ such that $ \Phi|_{X^{\prime}}(f) = g.$ So $\Phi|_{X^{\prime}}$ is surjective and hence $\Phi|_{X^{\prime}}$  is bijective.  
\end{proof}

\begin{corollary}\label{cor}
	Let $D$,  $D^{\prime}$ and $\Phi$   are same as defined in Notation \ref{phi.}. Let   ${I}$ and $\tilde{I}$ be the edge ideals of   $D$ and $D^{\prime},$  respectively. Then      $ {\tilde{I}}^{(s)} = {\tilde{I}}^s $  if and only if $ {I}^{(s)} = {I}^s $   for each $s \geq 1.$    
\end{corollary}

\begin{proof}  
By Proposition \ref{sdefect}, $\sdefect(\tilde{I},s) =0$ if and only if $ \sdefect(I,s)=0$ for each $s \geq 1.$ Hence the proof follows.  
\end{proof}

If all the vertices of   $~V^{+}(D)$ are sinks, we get the following two results. 

\begin{corollary}\label{cor2}  
Let $D$ be a  weighted oriented graph $ D $ where the vertices of   $~V^{+}(D)$ are sinks and its  underlying graph is $G$.  Then      $ {{I(G)}}^{(s)} = {{I(G)}}^s $  if and only if $ {I(D)}^{(s)} = {I(D)}^s $   for each $s \geq 1.$    .  	
\end{corollary}

\begin{proof}
Let $D^{\prime}$  is same as defined in Notation \ref{phi.}. Then $D^{\prime} = G$ and the proof  follows  from Corollary \ref{cor}.         
\end{proof}

\begin{corollary}\label{cor3} 
	Let $D$ be  a  weighted oriented graph  where the vertices of   $~V^{+}(D)$ are sinks and its underlying graph is $G$. Then   $G$ is bipartite if and only if $ {{I(D)}^{(s)}} = {{I(D)}^{s}} $ for all $ s \geq 2 $.    .  	
\end{corollary}  

\begin{proof}
It follows from \cite[Theorem 5.9]{simis} and Corollary \ref{cor2}.  		
%
\end{proof}    

\begin{remark}
As an application of Theorem \ref{sym.theorem.1} and Corollary \ref{cor}, by studying
the symbolic powers of edge ideals of one class of weighted oriented graphs, we can get
information about the symbolic powers of edge ideals of another class of weighted oriented
graphs. When we try to find the necessary and sufficient condition for the equality of
ordinary and symbolic powers of edge ideals of a certain class of weighted oriented
	graphs, as an application of Corollary \ref{cor}, we can omit the checking of equality of
	ordinary and symbolic powers of the edge ideals of those weighted oriented graphs where
	some vertex with non-trivial weight is sink. Hence we need to check the equality only
	for a smaller class of graphs.
\end{remark}

In the next result, as an application of Corollary \ref{cor}, we give necessary and sufficient
condition for the equality of  ordinary and symbolic powers of edge ideals of
weighted oriented even cycles of length $ 4 $ which are not naturally oriented.

	\begin{proposition}\label{evencycle4}
	Let    $D$ be  a weighted oriented even cycle which
	is not naturally oriented with underlying graph  $C_{4} = (x_1,x_2,x_3,x_{4})$  and at least one vertex of $D$ has non-trivial weight. Then $ I(D)^{(s)} = I(D)^s $ for all $s \geq 2$   if and only if $D$ is not of
	class $ (7) $ (See Figure \ref{fig.2}).
\end{proposition}

\begin{proof}
If all vertices of $ D $ have non-trivial weights, then $ D $ is naturally oriented. So
$ D $ has at most three vertices with non-trivial weights. Let $D^{\prime} $ is same as defined  in
Notation \ref{phi.}. Let $I= I(D)$ and $\tilde{I}= I(D^{\prime})$.  We check the equality of ordinary and symbolic powers by
considering different cases depending upon the number of vertices with non-trivial weights. We do not consider any weighted oriented even cycle which can be regarded as some weighted naturally oriented even cycle by changing the orientation of edges.

\textbf{Case (1)} $ D $ has only one vertex with non-trivial weight.

Then $ D $ is of class $ (1) $ (see Figure \ref{fig.2}). Here we can think $D^{\prime} $
as a simple bipartite graph. Hence by Corollary \ref{cor3},  we get $ I^{(s)} = I^s $ for all $s \geq 2$.


\begin{figure}[!ht]
	\begin{tikzpicture}[scale=0.85]
	\begin{scope}[ thick, every node/.style={sloped,allow upside down}] 
	\definecolor{ultramarine}{rgb}{0.07, 0.04, 0.56} 
	\definecolor{zaffre}{rgb}{0.0, 0.08, 0.66}   
	\draw[fill, black] (0,0) --node {\midarrow}(0,2);
	\draw[fill, black] (2,0) --node {\midarrow}(2,2);
	\draw[fill, black] (2,2) --node {\midarrow}(0,2);
	\draw[fill, black] (0,0) --node {\midarrow}(2,0);
	\draw [fill] [fill] (0,0) circle [radius=0.04];
	\draw[fill] [fill] (2,0) circle [radius=0.04];
	\draw[fill] [fill] (2,2) circle [radius=0.04];
	\draw[fill] [fill] (0,2) circle [radius=0.07];

	\node at (0,-0.4) {$1$};
	\node at (2,-0.4) {$1$};
	\node at (2,2.4) {$1$};
	\node at (0,2.4) {$w \neq 1$};
	
	\node at (1,-0.6) {$(1)$};
	
	
	\draw[fill, black] (4,0) --node {\midarrow}(4,2);
	\draw[fill, black] (4,2) --node {\midarrow}(6,2);
	\draw[fill, black] (6,0) --node {\midarrow}(6,2);
	\draw[fill, black] (4,0) --node {\midarrow}(6,0);
	\draw [fill] [fill] (4,0) circle [radius=0.04];
	\draw[fill] [fill] (6,0) circle [radius=0.04];
	\draw[fill] [fill] (6,2) circle [radius=0.07];
	\draw[fill] [fill] (4,2) circle [radius=0.04];

	\node at (4,-0.4) {$1$};
	\node at (6,-0.4) {$1$};
	\node at (6,2.4) {$w \neq 1$};
	\node at (4,2.4) {$w \neq 1$};

	\node at (5.1,-0.6) {$(2)$};
	
\draw[fill, black] (8,0) --node {\midarrow}(8,2);
\draw[fill, black] (8,2) --node {\midarrow}(10,2);
\draw[fill, black] (10,2) --node {\midarrow}(10,0);
\draw[fill, black] (8,0) --node {\midarrow}(10,0);
\draw [fill] [fill] (8,0) circle [radius=0.04];
\draw[fill] [fill] (10,0) circle [radius=0.07];
\draw[fill] [fill] (10,2) circle [radius=0.04];
\draw[fill] [fill] (8,2) circle [radius=0.04];

\node at (8,-0.4) {$1$};
\node at (10,-0.4) {$w \neq 1$};
\node at (10,2.4) {$1$};
\node at (8,2.4) {$w \neq 1$};

\node at (9.1,-0.6) {$(3)$};

\draw[fill, black] (12,0) --node {\midarrow}(12,2);
\draw[fill, black] (14,2) --node {\midarrow}(12,2);
\draw[fill, black] (14,2) --node {\midarrow}(14,0);
\draw[fill, black] (12,0) --node {\midarrow}(14,0);
\draw [fill] [fill] (12,0) circle [radius=0.04];
\draw[fill] [fill] (14,0) circle [radius=0.07];
\draw[fill] [fill] (14,2) circle [radius=0.04];
\draw[fill] [fill] (12,2) circle [radius=0.07];

\node at (12,-0.4) {$1$};
\node at (14,-0.4) {$w \neq 1$};
\node at (14,2.4) {$1$};
\node at (12,2.4) {$w \neq 1$};

\node at (13.1,-0.6) {$(4)$};

\draw[fill, black] (0,-4) --node {\midarrow}(0,-2);
\draw[fill, black] (2,-4) --node {\midarrow}(2,-2);
\draw[fill, black] (0,-2) --node {\midarrow}(2,-2);
\draw[fill, black] (0,-4) --node {\midarrow}(2,-4);
\draw [fill] [fill] (0,-4) circle [radius=0.04];
\draw[fill] [fill] (2,-4) circle [radius=0.04];
\draw[fill] [fill] (2,-2) circle [radius=0.04];
\draw[fill] [fill] (0,-2) circle [radius=0.04];

\node at (0,-4.4) {$1$};
\node at (2,-4.4) {$w \neq 1$};
\node at (2,-1.6) {$1$};
\node at (0,-1.6) {$w \neq 1$};

\node at (1,-4.7) {$(5)$};

	\draw[fill, black] (4,-4) --node {\midarrow}(4,-2);
\draw[fill, black] (6,-4) --node {\midarrow}(6,-2);
\draw[fill, black] (4,-2) --node {\midarrow}(6,-2);
\draw[fill, black] (4,-4) --node {\midarrow}(6,-4);
\draw [fill] [fill] (4,-4) circle [radius=0.04];
\draw[fill] [fill] (6,-4) circle [radius=0.04];
\draw[fill] [fill] (6,-2) circle [radius=0.07];
\draw[fill] [fill] (4,-2) circle [radius=0.04];

\node at (4,-4.4) {$1$};
\node at (6,-4.4) {$w \neq 1$};
\node at (6,-1.6) {$w \neq 1$};
\node at (4,-1.6) {$w \neq 1$};

\node at (5,-4.7) {$(6)$};

	\draw[fill, black] (8,-4) --node {\midarrow}(8,-2);
\draw[fill, black] (10,-2) --node {\midarrow}(10,-4);
\draw[fill, black] (8,-2) --node {\midarrow}(10,-2);
\draw[fill, black] (8,-4) --node {\midarrow}(10,-4);
\draw [fill] [fill] (8,-4) circle [radius=0.04];
\draw[fill] [fill] (10,-4) circle [radius=0.07];
\draw[fill] [fill] (10,-2) circle [radius=0.04];
\draw[fill] [fill] (8,-2) circle [radius=0.04];

\node at (8,-4.4) {$1$};
\node at (10,-4.4) {$w \neq 1$};
\node at (10,-1.6) {$w \neq 1$};
\node at (8,-1.6) {$w \neq 1$};

\node at (9,-4.7) {$(7)$};

	\end{scope}
	\end{tikzpicture}
	\caption{All classes of weighted oriented even cycles of length $ 4 $ which
		are not naturally oriented.}\label{fig.2}
\end{figure}
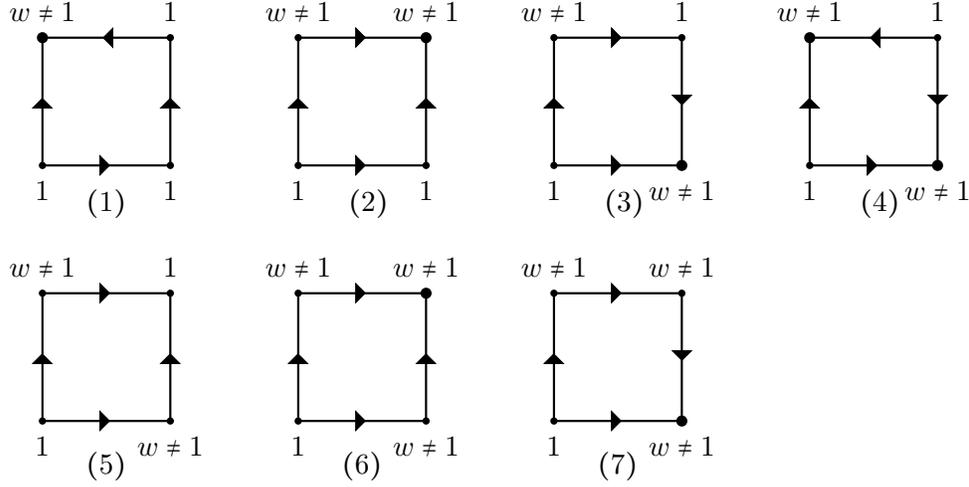

\textbf{Case (2)} $ D $ has only two vertices with non-trivial weights.

Then $ D $ is one of the classes $ (2) , (3), (4)$ and $ (5) $ (see Figure \ref{fig.2}). If $ D $ is of class (2), then we can think $D^{\prime} $
as a weighted naturally oriented even cycle where only one vertex has non-trivial weight. Using
Proposition \ref{evencycle2}, we get $ \tilde{I}^{(s)} = \tilde{I}^s $ for all $s \geq 2$. Hence by Corollary \ref{cor}, we have $ I^{(s)} = I^s $ for all $s \geq 2$. If $ D $ is of class $ (3) $, then by the similar argument as in class $ (2) $, $ I^{(s)} = I^s $ for all $s \geq 2$. If $ D $ is of class $ (4) $, then by the similar argument as in class $ (1) $, $ I^{(s)} = I^s $ for all $s \geq 2$. Now assume $ D $ is of class $ (5) $. There is no vertex with non-trivial weight which is sink. Without loss of generality we can assume that $ w(x_2) \neq 1$ and $ w(x_4) \neq 1 $.
Then $ I = (x_1x_2^{w_2}, x_2x_3, x_3x_4, x_1x_4^{w_4}).$  By the similar argument as in Remark \ref{one.weight}, we find 
$ I^{(s)} = (x_1, x_3)^s \bigcap ((x_2^{w_2}
, x_3, x_4^{w_4}) \cap (x_2, x_4))^s = (x_1, x_3)^s \cap (x_2^{w_2}
, x_2x_3, x_3x_4, x_4^{w_4})^s $. 
Let $ \bar{m} \in  \mathcal{G}(I^{(s)}). $ Then $ \bar{m} = \lcm(m_1,m_2) $ for some $ m_1  \in  \mathcal{G}((x_1, x_3)^s) $ and $ m_2  \in  \mathcal{G}((x_2^{w_2}
, x_2x_3, x_3x_4, x_4^{w_4})^s) $. Thus $ m_1 = x_1^{a_1}x_3^{a_2} $ and $ m_2 = (x_2^{w_2})^{b_1}(x_2x_3)^{b_2}(x_3x_4)^{b_3}(x_4^{w_4})^{b_4} $ for some
$ a_i, b_i \geq 0 $ with $ a_1 + a_2 = s $ and $ b_1 + b_2 + b_3 + b_4 = s. $


Assume that $ a_1 \geq b_1 + b_4.$

Then $ x_1^{a_1}(x_2^{w_2})^{b_1}(x_2x_3)^{b_2}(x_3x_4)^{b_3}(x_4^{w_4})^{b_4}
 \divides \lcm(m_1,m_2) = \bar{m}.$ Since $ x_1x_2^{w_2}
$ and $ x_1x_4^{w_4} \in 
\mathcal{G}(I), $ $ x_1^{a_1}(x_2^{w_2})^{b_1}(x_4^{w_4})^{b_4}$ can be expressed as a multiple of product of $ b_1 + b_4 $ elements of
$ \mathcal{G}(I). $ So $ \bar{m} \in I^{(b_2+b_3)+(b_1+b_4) }= I^s.$

Now assume that $ a_1 < b_1 + b_4.$ Then $ a_2 > b_2 + b_3.$

Thus $ \bar{m} = \lcm(m_1,m_2) = x_1^{a_1}x_3^{a_2-(b_2+b_3)}(x_2^{w_2})^{b_1}(x_2x_3)^{b_2}(x_3x_4)^{b_3}(x_4^{w_4})^{b_4} $.


 Here $ a_1 + a_2 -
(b_2 + b_3) = s - (b_2 + b_3) = b_1 + b_4.$ Thus $x_1^{a_1}x_3^{a_2-(b_2+b_3)}(x_2^{w_2})^{b_1}(x_4^{w_4})^{b_4}$ can be expressed as a multiple of product of $ b_1 + b_4 $ elements of
$ \mathcal{G}(I). $ So $ \bar{m} \in I^{(b_2+b_3)+(b_1+b_4) }= I^s.$

Hence $ I^{(s)} = I^s $ for all $s \geq 2$.

\textbf{Case (3)} $ D $ has only three vertices with non-trivial weights.

Then $ D $ is one of the classes $(6)$ and $ (7) $ (see Figure \ref{fig.2}). If $ D $ is of class $ (6) $, then  $D^{\prime} $ is of class $ (5) $. We know $ \tilde{I}^{(s)} = \tilde{I}^s $ for all $s \geq 2$. Hence by Corollary \ref{cor}, we have $ I^{(s)} = I^s $ for all $s \geq 2$. If $ D $ is of class $ (7) $,
then we can think $D^{\prime} $ as a weighted naturally oriented even cycle where only two consecutive vertices have non-trivial weights. By 
Proposition \ref{evencycle2}, we have $ \tilde{I}^{(s)} \neq \tilde{I}^s $  for some $s \geq 2$. Then by Corollary \ref{cor},  we get $ I^{(s)} \neq I^s $. Hence the proof follows.	
\end{proof}

Next we see another application of Corollary \ref{cor}  to weighted oriented star graphs.
\begin{definition}
A star graph $ S_n $ of order $ n $ is a tree on $n+1$ vertices with one vertex having  degree $n$ and the other $ n $ vertices having  degree $1.$     
\end{definition}

In the next theorem, we show that the  ordinary and
symbolic powers of edge ideal of any weighted  oriented star graph are equal.
\begin{theorem}\label{stargraph}  
	Let $D$  be  a weighted oriented star graph with underlying graph is $S_n$ for some $n \geq 2$. Then  $ I(D)^{(s)} = I(D)^s $ for all $s \geq 2.$          
\end{theorem}

\begin{proof}
	
	Let $V(D) =   \{x_0,x_1,\ldots,x_n     \}$ with $\deg_D(x_0) = n$ and $\deg_D(x_i)=1$ if $i \neq 0$. Here $E(S_n) =  \{ \{x_0,x_1   \},\{x_0,x_2   \},\ldots,\{x_0,x_n\} \}$.

	\textbf{Case (1)} Assume that $w(x_0) = 1.$
	
	If $w(x_i) \neq 1$ for some $i \neq 0$, then $(x_0,x_i)$   $\in E(D).$ This implies  $x_i$ is a sink vertex. So all vertices of $V^+(D)$ are sinks. Thus by Corollary \ref{cor3}, we have  $ I(D)^{(s)} = I(D)^s $  for all $s \geq 2.$

	\textbf{Case (2)} Assume that $w(x_0) \neq 1,$ i.e., $x_0 \in V^+(D).$
	
	Then $x_0$ is not a source vertex, i.e., $ N_D^-(x_0) \neq \phi. $

	$\textbf{\underline{\mbox{Case  (2.a)}}}$ Suppose  $ N_D^+(x_0) = \phi. $
	
	Then $x_0$ is a sink vertex. This implies that each $x_i$ for  $i \neq 0$ is a source vertex. So $w_i  = 1  $ for each   $i \neq 0$. Hence the only vertex of $V^+(D)$  is $x_0$ and it is sink. Then by Corollary \ref{cor3},  we have  $ I(D)^{(s)} = I(D)^s $  for all $s \geq 2.$

	$\textbf{\underline{\mbox{Case  (2.b)}}}$ Suppose  $ N_D^+(x_0) \neq \phi. $
	
	Without loss of generality we can assume that  $ N_D^+(x_0) = \{ x_1,x_2,\ldots,x_r \} $ and $ N_D^-(x_0) = \{ x_{r+1},x_{r+2},\ldots,x_n\}$ for some $r \geq 1$. If $ x_i \in V^+(D) $ for some $ i \in  [r], $ then $ x_i $ is sink. Let $D^{\prime} $ is same as defined  in
	Notation \ref{phi.}. Then $w_i=1$ for $1 \leq i \leq n$  in $D^{\prime}  $. Let $\tilde{I}= I(D^{\prime})$.
    By Corollary \ref{cor}, it is enough to show that $\tilde{I}^{(s)} = \tilde{I}^s $ for all $s \geq 2$.
   	
	 Note that the two minimal vertex covers of $D^{\prime}$ are $\{x_0\}$ and $\{x_1,\ldots,x_n     \}$. By Lemma \ref{minimal to strong}, these are strong vertex covers of $D^{\prime}.$ 
	
	Consider a vertex cover $C= \{x_0,x_1,x_2,\ldots,x_r\}$. Then $C^c =  \{ x_{r+1},x_{r+2},\ldots,x_n\}$. Here $ N_{D^{\prime}}^+(x_0) \cap C^c = \phi $ and $ N_{D^{\prime}}^-(x_0) \cap C^c \neq \phi  $. So $x_0 \in L_2^{D^{\prime}}(C)$.
	By Lemma \ref{L3}, $  L_3^{D^{\prime}}(C) = \{ x_1,x_2,\ldots,x_r \}.$ Since  $ x_0 \in  N_{D^{\prime}}^-(x_i) \cap V^+({D^{\prime}}) \cap L_2^{D^{\prime}}(C)$ for $1 \leq i \leq r,$ 
	$C$ is a strong vertex cover of ${D^{\prime}}$.
	
Consider $C^{\prime} \subsetneq C$ as a  vertex cover of ${D^{\prime}}.$ Since $x_{r+1} \notin C^{\prime},$  $x_0 \in C^{\prime}$. Thus there exist $j$ and $k \in [r]$ such that $x_j\in C^{\prime}$ and $x_k \notin C^{\prime}.$ Without loss of generality $x_j = x_1$ and $x_k = x_2.$ By Lemma \ref{L3},   $x_{1} \in  L_3^{D^{\prime}}(C^{\prime})$. Here $N_{D^{\prime}}^-(x_1) = \{x_0\}  \subseteq L_1^{D^{\prime}}(C^{\prime})$ because $  x_2 \in  N_{D^{\prime}}^+(x_0) \cap {C^{\prime}}^c.$ Hence by Remark \ref{s.v.1}, $C^{\prime}$ is not strong. 
	
Let $C_1= \{x_0\}$, $C_2 = \{x_1,\ldots,x_n     \}$ and $C_3 = \{x_0,x_1,\ldots,x_r     \}$.    
Suppose  there exists a  strong vertex cover $C_4$ of ${D^{\prime}}$ other than $C_1,C_2$ and $C_3$. We know that any vertex cover which is a proper subset of $C_3$ can not be strong.
Thus $C_4$ must contain $x_0$ and some vertex $x_i \in \{ x_{r+1},x_{r+2},\ldots,x_n\}$. Without loss of generality we can assume that $C_4$  contains $x_0$ and  $x_{r+1}$.  
	By Lemma \ref{L3},   $x_{r+1} \in  L_3^{D^{\prime}}(C_4) $. Since $N_{D^{\prime}}^-(x_{r+1}) = \phi  ,$ by Remark \ref{s.v.1}, $C_4$ is not a strong vertex cover of ${D^{\prime}}$. Hence $C_1,C_2$ and $C_3$ are the only  strong vertex covers of ${D^{\prime}}.$ Consider  $C_1= \{x_0\}.$ Here $   x_1 \in  N_{D^{\prime}}^+(x_0) \cap C_1^c    $. So  $ L_1^{D^{\prime}}(C_1) = \{x_0\}$. Hence $\tilde{I}_{C_1} =(x_0).$
	Consider  $C_2 = \{x_1,\ldots,x_n     \}.$
	Then   $\tilde{I}_{C_2} =( x_1,\ldots, x_r, x_{r+1},\ldots,x_n ).$
	Consider   $C_3 = \{x_0,x_1,\ldots,x_r     \}.$ We know $  L_2^{D^{\prime}}(C_3) = \{ x_0 \} $. This implies $\tilde{I}_{C_3} =( x_0^{w_0}, x_1,\ldots, x_r ).$
	Hence by Lemma \ref{cooper}, we have
\begin{align*}
\tilde{I}^{(s)} &= ( ( x_0^{w_0}, x_1,\ldots, x_r )\cap(x_0)   )^s \bigcap ( x_1,\ldots, x_r, x_{r+1},\ldots,x_n  )^s\\
&= (  x_0^{w_0}, x_0x_1,\ldots, x_0x_r )^s \bigcap ( x_1,\ldots, x_r, x_{r+1},\ldots,x_n  )^s. 
\end{align*}
 Let $\bar{m} \in \mathcal{G}(\tilde{I}^{(s)} )$. Then $\bar{m} = \lcm(m_1,m_2)$ for some $m_1 \in \mathcal{G}((  x_0^{w_0}, x_0x_1,\ldots ,    x_0x_r )^s  )$ and\\ $m_2 \in \mathcal{G}(( x_1,\ldots ,   x_r, x_{r+1},\ldots,x_n  )^s ).$ Thus   $ m_1 =   (x_0^{w_0})^{a_0} (x_0x_1)^{a_1}\cdots (x_0x_r)^{a_r}$   and  $ m_2 =  x_1^{b_1} \cdots   x_r^{b_r} x_{r+1}^{b_{r+1}}  \cdots x_n^{b_n}    $ for some $a_i, b_i \geq 0$ with   $\displaystyle{ \sum_{i=0}^{r}a_i  = s}$ and $\displaystyle{ \sum_{i=1}^{n}b_i  = s}.$

If $a_0 = 0,$ then   $m_1 \in \tilde{I}^s$ and so   $\bar{m} \in \tilde{I}^s$.
	
Now we assume that $a_0 \neq 0$ and $b_{r+1}+\cdots+b_n \geq  a_0.$ Here $m_1x_{r+1}^{b_{r+1}}  \cdots x_n^{b_n}\divides\lcm(m_1,m_2) =\bar{m}$.
	Then we can express   $ 	m_1x_{r+1}^{b_{r+1}}  \cdots x_n^{b_n}= (x_0x_1)^{a_1}\ldots (x_0x_r)^{a_r} [(x_0^{w_0})^{a_0}x_{r+1}^{b_{r+1}}  \cdots x_n^{b_n}]. $ Since $x_{i}x_0^{w_0} \in \mathcal{G}(\tilde{I})$ for $r+1 \leq i \leq n,$ $(x_0^{w_0})^{a_0}x_{r+1}^{b_{r+1}}  \cdots x_n^{b_n}$ can be expressed as a multiple of product of $a_0$ elements of $\mathcal{G}(\tilde{I})$. So $m_1x_{r+1}^{b_{r+1}}  \cdots x_n^{b_n} \in I^{(a_1+a_2+\cdots+a_r)+a_0} = \tilde{I}^s$. Hence $\bar{m} \in \tilde{I}^s$.

	Finally, we assume that $a_0 \neq 0$ and $b_{r+1}+\cdots+b_n <  a_0.$ Here $ (x_0^{w_0})^{a_0}x_0^{a_1 +\cdots+a_r} m_2 \divides\lcm(m_1,m_2) = \bar{m}$. Then we can express    
	\begin{align*}
\hspace*{0.5cm}	&\hspace*{0.43cm}(x_0^{w_0})^{a_0}x_0^{a_1 +\cdots+a_r} m_2\\
	&=(x_0^{w_0})^{a_0}x_0^{a_1 +\cdots+a_r} x_1^{b_1} \cdots  x_r^{b_r} x_{r+1}^{b_{r+1}}  \cdots x_n^{b_n}\\
	&= (x_0^{w_0})^{a_0 -(b_{r+1}+\cdots+b_n  )  }x_0^{a_1 +\cdots+a_r} x_1^{b_1} \cdots  x_r^{b_r} [x_{r+1}^{b_{r+1}}  \cdots x_n^{b_n}(x_0^{w_0})^{(b_{r+1}+\cdots+b_n  )  }] \\
	&= (x_0^{w_0-1})^{a_0 -(b_{r+1}+\cdots+b_n  )  }[x_0^{a_0-(b_{r+1}+\cdots+b_n  ) +a_1 +\cdots+a_r} x_1^{b_1} \cdots  x_r^{b_r}] [(x_{r+1}x_0^{w_0})^{b_{r+1}} \cdots (x_nx_0^{w_0})^{b_n}]\\
	&= (x_0^{w_0-1})^{a_0 -(b_{r+1}+\cdots+b_n  )      }[x_0^{b_1 +\cdots+b_r} x_1^{b_1} \cdots  x_r^{b_r}] [(x_{r+1}x_0^{w_0})^{b_{r+1}} \cdots (x_nx_0^{w_0})^{b_n}]\\
	&= (x_0^{w_0-1})^{a_0 -(b_{r+1}+\cdots+b_n  )      } [(x_0x_1)^{b_1} \cdots  (x_0 x_r)^{b_r}] [(x_{r+1}x_0^{w_0})^{b_{r+1}} \cdots (x_nx_0^{w_0})^{b_n}]. 
	\end{align*} 
	So     $ (x_0^{w_0})^{a_0}x_0^{a_1 +\cdots+a_r} m_2 \in \tilde{I}^{(b_1   +\cdots+b_r) + (b_{r+1} + \cdots+ b_n)} =   \tilde{I}^s$. Therefore $\bar{m} \in \tilde{I}^s$. Hence the proof follows.
\end{proof}


\begin{thebibliography}{AAAA}

\bibitem{atiyah}  	
{M.F. Atiyah and I.G. Macdonald}, {Introduction to
Commutative Algebra}, Addison-Wesley (1969).


\bibitem{kanoy}
{A. Banerjee,
B. Chakraborty,
K. K. Das,
M. Mandal and S. Selvaraja}, {Equality of ordinary and symbolic powers of edge ideals of weighted oriented graphs}, Preprint (2021).


	
  

	
%
	
%
	     
	
	
	\bibitem{cooper}	
	{S. Cooper, R. Embree, H. T.  H{\`a} and A. H. Hoefel}, {Symbolic powers of monomial ideals},\textit{ Proc. Edinb. Math. Soc.} (2) {\bf60} (2017), no. 1, 39–55.


\bibitem{huneke}
{H. Dao, A. De Stefani, E. Grifo, C. Huneke, and L. N{\'u}{\~n}ez Betancourt}, {Symbolic powers of ideals}, {In
Singularities and foliations. geometry, topology and applications},  \textit{Springer Proc. Math. Stat.} {\bf222} (2018), Springer, Cham, 387–432.

	


\bibitem{gimenez}
{P. Gimenez, J. M. Bernal, A. Simis, R. H. Villarreal, and C. E. Vivares},{ Symbolic powers of monomial ideals and Cohen-Macaulay vertex-weighted digraphs}.  \textit{Singularities, algebraic geometry, commutative algebra, and related topics}, 491–510, Springer, Cham, 2018.

\bibitem{hochster}  
{M. Hochster}, {Criteria for equality of ordinary and symbolic powers of primes}, \textit{Math. Z.} {\bf133} (1973), 53–65. 
	

\bibitem{li}
{A. Li and I. Swanson}, {Symbolic powers of radical ideals}, \textit{Rocky Mountain J. Math.} {\bf36} (2006), no. 3, 997–1009.  	

	
	
	
	
	
\bibitem{mandal1}
{M. Mandal and D. K. Pradhan}, {Symbolic powers in weighted oriented graphs}, \textit{Internat. J. Algebra
	Comput.} {\bf31} (2021), no. 03,  533-549. 
  

		
	
	
\bibitem{pitones}
{Y. Pitones, E. Reyes and J. Toledo}, {Monomial ideals of weighted oriented graphs}, \textit{Electron. J. Combin.} {\bf26} (2019), no. 3, 1-18.  

\bibitem{simis}
{A. Simis, W. Vasconcelos and R. H. Villarreal, On the ideal theory of graphs}, \textit{J. Algebra} {\bf167} (1994),
389–416.


	
	
%
%
%


%
%
%
%
	
	
\end{thebibliography}
\end{document}